\documentclass{amsart} 
\usepackage{a4wide}
\usepackage{color}
\usepackage{amssymb}
\usepackage{mathrsfs}

\newtheorem{theorem}{Theorem}[section]
\newtheorem{lemma}[theorem]{Lemma}
\newtheorem{proposition}[theorem]{Proposition}
\newtheorem{corollary}[theorem]{Corollary}

\newtheorem{problem}[theorem]{Problem}

\theoremstyle{definition}
\newtheorem{definition}[theorem]{Definition}
\newtheorem{example}[theorem]{Example}

\theoremstyle{remark}
\newtheorem{remark}[theorem]{Remark}

\numberwithin{equation}{section}

\usepackage{epsfig}
\usepackage[all]{xy}

\newcommand{\C}{{\mathbb{C}}}

\newcommand{\Z}{{\mathbb{Z}}}

\newcommand{\LL}{{\mathbb{L}}}

\newcommand{\algt}{\mathfrak{t}}

\begin{document}

\title[Maximal torus actions on GKM manifolds]{Upper bounds for the dimension of tori acting on GKM manifolds}
\date{\today}

\author[S.\ Kuroki]{Shintar\^o Kuroki}
\address{Graduate School of Mathematical Sciences, The University of Tokyo, 3-8-1 Komaba, Meguro-ku, Tokyo, 153-8914, Tokyo, Japan}
\email{kuroki@ms.u-tokyo.ac.jp}



\begin{abstract}
The aim of this paper is to give an upper bound for the dimension of a torus $T$ which acts on a GKM manifold $M$ effectively.
In order to do that, we introduce a free abelian group of finite rank, denoted by $\mathcal{A}(\Gamma,\alpha,\nabla)$, 
from an (abstract) $(m,n)$-type GKM graph $(\Gamma,\alpha,\nabla)$.
Here, an $(m,n)$-type GKM graph is the GKM graph induced from a $2m$-dimensional GKM manifold $M^{2m}$ with an effective $n$-dimensional torus $T^{n}$-action which preserves the almost complex structure, say $(M^{2m},T^{n})$.
Then it is shown that $\mathcal{A}(\Gamma,\alpha,\nabla)$ has rank $\ell(> n)$ if and only if there exists an $(m,\ell)$-type GKM graph 
$(\Gamma,\widetilde{\alpha},\nabla)$ which is an extension of $(\Gamma,\alpha,\nabla)$.
Using this combinatorial necessarily and sufficient condition, 
we prove that the rank of $\mathcal{A}(\Gamma_{M},\alpha_{M},\nabla_{M})$ for the GKM graph $(\Gamma_{M},\alpha_{M},\nabla_{M})$ induced from $(M^{2m},T^{n})$ gives an upper bound for the dimension of a torus which can act on $M^{2m}$ effectively.
As one of applications of this result, we compute the rank of $\mathcal{A}(\Gamma,\alpha,\nabla)$ of the complex Grassmannian of $2$-planes $G_{2}(\mathbb{C}^{n+2})$ 
with some effective $T^{n+1}$-action, 
and prove that the $T^{n+1}$-action on $G_{2}(\mathbb{C}^{n+2})$ is the maximal effective torus action which preserves the standard complex structure.
\end{abstract}

\maketitle

\section{Introduction}
\label{sect:1}

GKM manifolds (or more general spaces) are -roughly-  
spaces with torus action whose 0- and 1-dimensional orbits have the structure of a graph.  
This class of spaces was first appeared in the work of Goresky-Kottwitz-MacPherson \cite{GKM} as a class of algebraic varieties (GKM stands for their initials).
Motivated by their works, Guillemin-Zara \cite{GuZa} introduce a combinatorial counterpart of the GKM manifold, called a(n) (abstract) GKM graph, 
and give some dictionary between the (symplectic) geometry (and topology) of GKM manifolds and 
the combinatorics of GKM graphs.
This leads us to the study of geometric and topological properties of GKM manifolds using combinatorial properties of GKM graphs (see e.g. \cite{FIM, GW, GHZ, GSZ, Ku15, MMP, Ta, Ty} etc).
In this paper, we introduce a new invariant of GKM graphs and provide a partial answer to the extension problem of torus actions on GKM manifolds. 

To state our main results more precisely, we briefly recall the setting of this paper and background of the extension problem of torus actions.
Let $T^{n}$ be the $n$-dimensional torus and $M^{2m}$ be a $2m$-dimensional, compact, connected, almost complex manifold with effective $T^{n}$-action, which preserves the almost complex structure.
We denote such manifold as $(M^{2m}, T^{n})$, or $M^{2m}$, $M$, $(M,T)$ (if its torus action or dimensions of a manifold and a torus are obviously known from the context).
We call $(M^{2m}, T^{n})$ a {\it GKM manifold} if it satisfies the following properties (see Section~\ref{sect:4} for details):
\begin{enumerate}
\item the set of fixed points is not empty and isolated, i.e., $M^{T}$ is $0$-dimensional;
\item the closure of each connected component of $1$-dimensional orbits is equivariantly diffeomorphic to the $2$-dimensional sphere, called an {\it invariant $2$-sphere}. 
\end{enumerate}
Regarding fixed points as vertices and invariant $2$-spheres as edges, 
this condition is equivalent to that the one-skeleton of $(M^{2m}, T^{n})$ has the structure of a graph, where 
a {\it one-skeleton} of $(M^{2m}, T^{n})$ is the orbit space of the set of 0- and 1-dimensional orbits.  
Note that there are several definitions of GKM manifolds (see e.g. \cite{GHZ, GuZa} etc).
This is because the spaces with such torus actions
have been studied from several different points of view (homotopically, topologically, algebraically or geometrically).
In this paper, we study the GKM manifolds defined by Guillemin-Zara in their original paper \cite{GuZa}.
For example, 
in our setting, the following manifolds are GKM manifolds: non-singular complete toric varieties (also called toric manifolds) and 
homogeneous manifolds $G/H$ (where $G$ is a compact connected Lie group and $H$ is its closed subgroup with the same maximal torus) with torus invariant almost complex structures such as $S^6=G_{2}/SU(3)$, flag manifolds and complex Grassmannians, etc.

Because GKM manifolds are even-dimensional and their effective torus actions have isolated fixed points, 
the differentiable slice theorem tells us that the following inequality holds for every GKM manifold $(M,T)$:
\begin{align*}
\dim T\le \frac{1}{2}\dim M.
\end{align*}
If the equality $\dim T=\frac{1}{2}\dim M$ holds, such a GKM manifold is also known as a {\it torus manifold} (with invariant almost complex structure); famous examples are toric manifolds. 
Namely, by definition, the torus action on a torus manifold is maximal, i.e., the torus action can not be extended to a bigger torus action.
In this case, 
the author, Masuda and Wiemeler \cite{Ku10, Ku11, KuMa, Wie} study the extended $G$-actions of $T$-actions on torus manifolds, 
where $G$ is a non-abelian, compact, simply connected Lie group with the maximal torus $T$. 
On the other hand, for general GKM manifolds, the given torus action might not be maximal.
In fact, a restricted $T^{m-1}$-action of a $2m$-dimensional toric manifold $(M^{2m},T^{m})$ is often a GKM manifold $(M^{2m},T^{m-1})$,
in other words, this GKM manifold $(M^{2m},T^{m-1})$ extends to a toric manifold $(M^{2m},T^{m})$.
Thus the following problem naturally arises in the GKM manifolds (see also Proposition~\ref{ext-cat-in-GKM}):
\begin{problem}
\label{prob}
When does a GKM manifold $(M^{2m},T^{n})$ extend to a GKM manifold $(M^{2m},T^{\ell})$?
Here, $T^{n}\subset T^{\ell}$ and $n<\ell\le m$.
\end{problem}

To give an answer to this problem, 
we introduce a free abelian group with finite rank $\mathcal{A}(\Gamma,\alpha,\nabla)$, 
called a {\it group of axial functions}, for the GKM graph $(\Gamma,\alpha,\nabla)$ in Section~\ref{sect:2}.
Here, a {\it GKM graph} is -roughly- the following triple (see Section~\ref{sect:2} for details): 
an $m$-valent graph $\Gamma$; 
a function $\alpha:E(\Gamma)\to H^{2}(BT^n)\simeq \mathbb{Z}^{n}$, called an {\it axial function}; 
and a collection $\nabla$ of some bijective maps between out-going edges on adjacent vertices, called a {\it connection}.
We call such a GKM graph an {\it $(m,n)$-type} GKM graph in this paper.
The main theorem of this paper can be stated as follows (see Sections~\ref{sect:2} and \ref{sect:3} for details):
\begin{theorem}
\label{main:1}
Let $(\Gamma,\alpha,\nabla)$ be an abstract $(m,n)$-type GKM graph. 
Then, the following two statements are equivalent:
\begin{enumerate}
\item ${\rm rk}\ \mathcal{A}(\Gamma,\alpha,\nabla)\ge \ell$ for some $n\le \ell \le m$;
\item there is an $(m,\ell)$-type GKM graph $(\Gamma,\widetilde{\alpha},\nabla)$ which is an extension of $(\Gamma,\alpha,\nabla)$.
\end{enumerate}
\end{theorem}

Because a GKM manifold $(M^{2m},T^{n})$ defines an $(m,n)$-type GKM graph (see Section~\ref{sect:4}), 
Theorem~\ref{main:1} implies that the maximal dimension of torus which can act on a GKM manifold $M$ is bounded from above by the rank of the  group of axial functions of the GKM graph induced from $M$.
Namely, we obtain the main result of this paper as follows (see Section~\ref{sect:4} for details):

\begin{corollary}
\label{main:2}
Let $(M^{2m},T^{n})$ be a GKM manifold and $(\Gamma_{M},\alpha_{M},\nabla_{M})$ be its $(m,n)$-type GKM graph.
Assume that ${\rm rk}\ \mathcal{A}(\Gamma_{M},\alpha_{M},\nabla_{M})=\ell$. 
Then, the $T^{n}$-action on $M^{2m}$ does not extend to any $T^{\ell+1}$-action preserving the given almost complex structure.

Hence, if ${\rm rk}\ \mathcal{A}(\Gamma_{M},\alpha_{M},\nabla_{M})=n$,  
then the $T^{n}$-action on $M^{2m}$ is maximal among torus actions which preserve the given almost complex structure.
\end{corollary}

\begin{remark}
Shunji Takuma also obtains a partial answer to Problem~\ref{prob} by introducing an obstruction class for the extension of an $(m,n)$-type GKM graph to an $(m,n+1)$-type GKM graph in his note \cite{Ta}.
Theorem~\ref{main:1} may be regarded as the generalization of his result.
\end{remark}

Problem~\ref{prob} is reminiscent of the computation of the {\it torus degree of symmetry} of a manifold $X$ (see \cite{Hs}), i.e., 
the maximal dimension of a torus which can act on $X$ effectively.
A torus degree of symmetry has been studied for many classes of manifolds, in particular from differential geometry (see e.g. \cite{ES, Hs, Wa, Wil}).
Corollary~\ref{main:2} may be regarded as to give an upper bound of the {\it torus degree of symmetry of an invariant almost complex structure} of 
a GKM manifold.
As an application of Corollary~\ref{main:2}, 
in the final section (Section~\ref{sect:5}), 
we compute the torus degree of such symmetry for the complex Grassmannian of $2$-planes, denoted as
\begin{align*}
G_{2}(\mathbb{C}^{n+2})\simeq GL(n+2,\mathbb{C})/GL(2,\mathbb{C})\times GL(n,\mathbb{C})\simeq U(n+2)/U(2)\times U(n).
\end{align*}
Namely, we compute ${\rm rk}\ \mathcal{A}(\Gamma_{M},\alpha_{M},\nabla_{M})$ for $M=G_{2}(\mathbb{C}^{n+2})$ with some effective $T^{n+1}$-action and prove the following fact:
\begin{proposition} 
\label{main:3}
The standard effective $T^{n+1}$-action on $G_{2}(\mathbb{C}^{n+2})$ is maximal among the effective torus actions which preserve the almost complex structure.
\end{proposition} 
Note that there is the natural $T^{n+2}$-action on $G_{2}(\mathbb{C}^{n+2})$ which is induced from the maximal torus subgroup in $U(n+2)$. 
However, this action is not effective.

The organization of this paper is as follows.
In Section~\ref{sect:2}, we recall an abstract GKM graph $(\Gamma,\alpha,\nabla)$, and introduce its group of axial functions $\mathcal{A}(\Gamma,\alpha,\nabla)$. 
In Section~\ref{sect:3}, the main theorem (Theorem~\ref{main:1}) is proved.
In Section~\ref{sect:4}, in order to apply our results to geometry, we recall the definition of a GKM graph induced from a GKM manifold, and show Corollary~\ref{main:2}.
We also prove the $T^{2}$-action on $S^{6}=G_{2}/SU(3)$ is maximal  in this section.
In Section~\ref{sect:5}, we obtain the GKM graph obtained from the effective $T^{n+1}$-action on $G_{2}(\mathbb{C}^{n+2})$, and 
compute its group of axial functions. 
This proves Proposition~\ref{main:3}.

\section{GKM graph and its group of axial functions}
\label{sect:2}

In this section, we first recall the basic facts about GKM graphs $(\Gamma,\alpha,\nabla)$ (see~\cite{GuZa}) 
and introduce the extension of axial functions of GKM graphs precisely.
Then, a finite rank free abelian group $\mathcal{A}(\Gamma,\alpha,\nabla)$, called a {\it group of axial functions}, is defined.

\subsection{GKM graph and its extension}
\label{sect:2.1}

We first prepare some notation to define a GKM graph.
Let $\Gamma=(V(\Gamma),E(\Gamma))$ be an {\it abstract graph} comprising a set $V(\Gamma)$ of vertices and a set  
$E(\Gamma)$ of oriented edges.
For the given orientation on $e\in E(\Gamma)$, we denote 
its initial vertex by $i(e)$ and its terminal vertex by $t(e)$.
In this paper, we assume that there are no loops in $E(\Gamma)$, i.e., $i(e)\not=t(e)$ for any $e\in E(\Gamma)$, and 
$\Gamma$ is connected.
The symbol $\overline{e}\in E(\Gamma)$ represents the edge $e$ with its orientation reversed, i.e., 
$i(e)=t(\overline{e})$ and $t(e)=i(\overline{e})$.
The subset $E_{p}(\Gamma)\subset E(\Gamma)$ is the set of out-going edges from $p\in V(\Gamma)$; more precisely,  
\[
E_{p}(\Gamma)=\{e\in E(\Gamma) \mid i(e)=p\}.
\]
A finite connected graph $\Gamma$ is called an {\it $m$-valent graph} if $|E_{p}(\Gamma)|=m$ for all $p\in V(\Gamma)$,
where the symbol $|X|$ represents the cardinality of a finite set $X$.

Let $\Gamma$ be an $m$-valent graph.
We next define a label $\alpha:E(\Gamma)\to H^{2}(BT)$ on $\Gamma$.
Recall that $BT^{n}$ (often denoted by $BT$) is a classifying space of an $n$-dimensional torus $T$,
and its cohomology ring (over $\mathbb{Z}$-coefficient) is isomorphic to the polynomial ring
\[
H^{*}(BT)\simeq \Z[a_{1},\ldots,a_{n}], 
\] 
where $a_{i}$ is a variable with $\deg a_{i}=2$ for $i=1,\ldots,n$.
So its degree $2$ part $H^{2}(BT)$ is isomorphic to $\mathbb{Z}^{n}$.
Put a label by a function $\alpha:E(\Gamma)\to H^{2}(BT)$ on edges of $\Gamma$.
Set
\[
\alpha_{(p)}=\{\alpha(e) \mid e\in E_{p}(\Gamma)\}\subset H^{2}(BT).
\]
An {\it axial function} on $\Gamma$ is the function $\alpha:E(\Gamma)\to H^{2}(BT^{n})$ for $n\le m$ 
which satisfies the following three conditions:
\begin{description}
\item[(1)] $\alpha(e)=-\alpha(\overline{e})$;
\item[(2)] for each $p\in V(\Gamma)$, the set $\alpha_{(p)}$ is {\it pairwise linearly independent}, i.e., 
each pair of elements in $\alpha_{(p)}$ is linearly independent in $H^{2}(BT)$;
\item[(3)] for all $e\in E(\Gamma)$, there exists a bijective map $\nabla_{e}:E_{i(e)}(\Gamma)\to E_{t(e)}(\Gamma)$ such that 
\begin{enumerate}
\item $\nabla_{\overline{e}}=\nabla_{e}^{-1}$, 
\item $\nabla_{e}(e)=\overline{e}$, and 
\item for each $e'\in E_{i(e)}(\Gamma)$,
there exists an integer $c_{e}(e')$ such that 
\begin{align}
\label{cong-rel}
\alpha(\nabla_{e}(e'))-\alpha(e')=c_{e}(e') \alpha(e)\in H^{2}(BT).
\end{align}
\end{enumerate} 
\end{description}
The collection $\nabla=\{\nabla_{e}\ |\ e\in E(\Gamma)\}$ is called a {\it connection} on the labelled graph $(\Gamma,\alpha)$;
we denote the labelled graph with connection as $(\Gamma,\alpha,\nabla)$,
and the equation \eqref{cong-rel} is called a {\it congruence relation}. 
We call the integer $c_{e}(e')$ in the congruence relation a {\it congruence coefficient of $e'$ on $e$}.
The conditions as above are called an {\it axiom of axial function}.
In addtion, in this paper, we also assume the followings:
\begin{description}
\item[(4)] for each $p\in V(\Gamma)$, the set $\alpha_{(p)}$ spans $H^{2}(BT)$.
\end{description}
The axial function which satisfies (4) is called an {\it effective} axial function.
\begin{definition}[GKM graph \cite{GuZa}]
\label{def:2-1}
If an $m$-valent graph $\Gamma$ is labeled by an axial function $\alpha:E(\Gamma)\to H^{2}(BT^{n})$ for some $n\le m$,
then such labeled graph is said to be an (abstract) {\it GKM graph}, and denoted as 
$(\Gamma,\alpha,\nabla)$ (or $(\Gamma,\alpha)$ if the connection $\nabla$ is obviously determined).
\end{definition}

\begin{definition}[$(m,n)$-type GKM graph]
\label{def:2-2}
Let $(\Gamma,\alpha, \nabla)$ be an abstract GKM graph.
If the axial function $\alpha$ is effective, $(\Gamma,\alpha, \nabla)$ is said to be an {\it $(m,n)$-type GKM graph}.
\end{definition}
In this paper, we only consider $(m,n)$-type GKM graphs ($n\le m$) unless otherwise stated.

Figure \ref{examples} shows examples of GKM graphs.
\begin{figure}[h]
\centering
\includegraphics[width=350pt,clip]{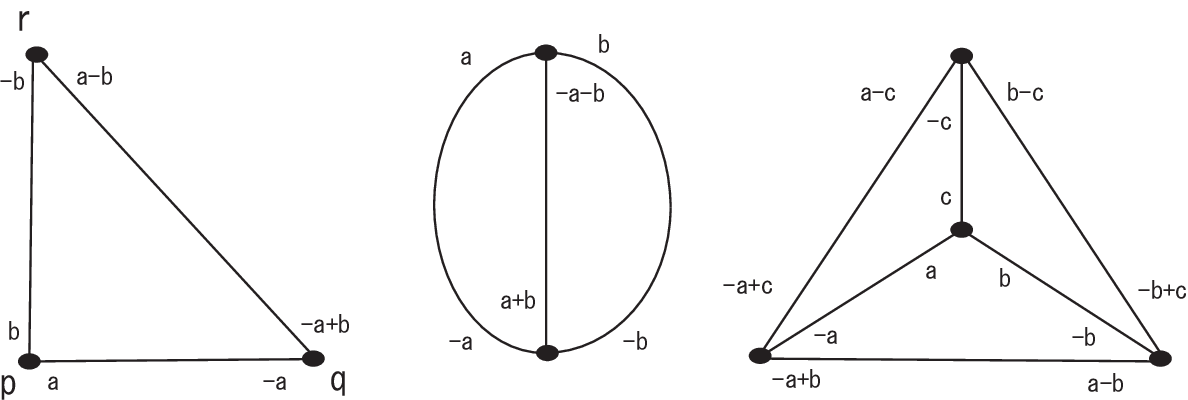}
\caption{Examples of (2,2)-type, (3,2)-type and (3,3)-type GKM graphs from left, where 
$a,b$ (resp.~$c$) are generators of $H^{*}(BT^{2})$ (resp.~$H^{*}(BT^{3})$). 
For example, in the left $(2,2)$-GKM graph, the axial function is defined by  $\alpha(pq)=a$, $\alpha(qr)=-a+b$, etc. In the (3,3)-type GKM graph, we omit the axial functions of the opposite directions of edges because it is automatically determined from the axiom (1) of GKM graph.}
\label{examples}
\end{figure}

We note the following lemma proved in~\cite[Proposition 2.1.3]{GuZa}.
\begin{lemma}
\label{lem:2-3}
Let $(\Gamma,\alpha,\nabla)$ be a GKM graph.
If $\alpha_{(p)}$ is three-independent for every $p\in V(\Gamma)$, 
the connection $\nabla$ is uniquely determined. 
\end{lemma}
Here, in Lemma~\ref{lem:2-3}, the set of vectors $\alpha_{(p)}$ in $H^{2}(BT)$ is called a {\it three-independent} if  
every triple $\{\alpha(e_{i}),\alpha(e_{j}),\alpha(e_{k})\}\subset \alpha_{(p)}$ is linearly independent 
(e.g., the right (3,3)-type GKM graph in Figure~\ref{examples}).
So, in this case, we may denote $(\Gamma,\alpha,\nabla)$ as $(\Gamma,\alpha)$ without connection $\nabla$.

We also note the following lemma:
\begin{lemma}
\label{lem:2-4}
For all $e\in E(\Gamma)$, $c_{e}(e)=-2$.
\end{lemma}
\begin{proof}
By the axiom (1), (3)-(2) and (3)-(3) of axial function, it is straightforward.
\end{proof}

We close this section by defining an extension.
Let $(\Gamma,\alpha,\nabla)$ be an $(m,n)$-type GKM graph.
An $(m,\ell)$-type GKM graph $(\widetilde{\Gamma},\widetilde{\alpha},\widetilde{\nabla})$ (for $n<\ell\le m$) is said to be an {\it extension} of $(\Gamma,\alpha,\nabla)$ if $\widetilde{\Gamma}=\Gamma$, $\nabla=\widetilde{\nabla}$ and there exists a 
projection $\pi:H^{2}(BT^{\ell})\to H^{2}(BT^{n})$ such that the following diagram commutes:
\begin{align*}
\xymatrix{
 & H^{2}(BT^{\ell})\ar[d]^{\pi}    \\
E(\Gamma) \ar[ru]^{\widetilde{\alpha}} \ar[r]^{\alpha} & H^{2}(BT^{n})  
}
\end{align*} 
Figure~\ref{eg-extend} shows an example of extensions.
\begin{figure}[h]
\centering
\includegraphics[width=300pt,clip]{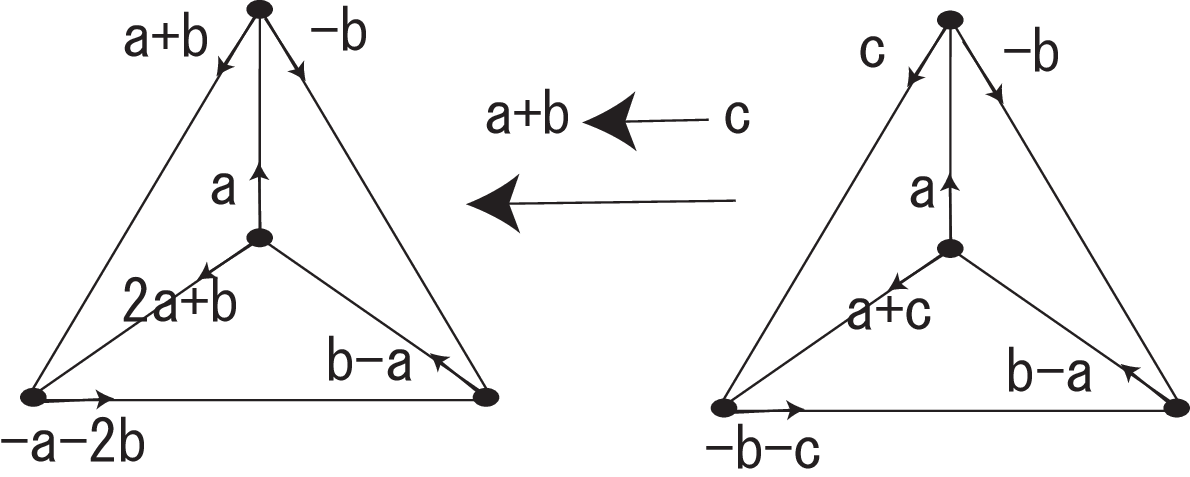}
\caption{The left (3,2)-type GKM graph extends to the right (3,3)-type GKM graph.
In this figure, we omit the axial function on the direction $\overline{e}$, because of the relation $\alpha(\overline{e})=-\alpha(e)$.}
\label{eg-extend}
\end{figure}

\subsection{The invariant function}
\label{sect:2.2}

Let $(\Gamma,\alpha,\nabla)$ be an $(m,n)$-type GKM graph for $n\le m$.
We shall define an invariant function $c_{(\Gamma,\alpha,\nabla)}:E(\Gamma)\to \mathbb{Z}^{m}$ under extensions.
To define it, we first fix an order of out-going edges on each vertex $p$, i.e., set 
\[
E_{p}(\Gamma)=\{e_{1,p},\ldots,e_{m,p}\}.
\]
Then, we can define the free $\mathbb{Z}$-module with rank $m$ on each $p$, say $\mathbb{Z}E_{p}(\Gamma)$, by regarding $\{e_{1,p},\ldots,e_{m,p}\}$ as the formal generator of $\mathbb{Z}E_{p}(\Gamma)$.
Namely, 
\[
\mathbb{Z}E_{p}(\Gamma):=\mathbb{Z}e_{1,p}\oplus \cdots\oplus \mathbb{Z}e_{m,p}\simeq \mathbb{Z}^{m}.
\]
Because an order on each $E_{p}(\Gamma)$ is fixed, 
the connection $\nabla_{e}:E_{i(e)}(\Gamma)\to E_{t(e)}(\Gamma)$ induces the permutation on $\{1,\ldots,m\}$.
So, there is a permutation $\sigma:\{1,\ldots,m\}\to \{1,\ldots,m\}$ such that 
\begin{align*}
\nabla_{e}(e_{j,i(e)})=e_{\sigma(j),t(e)}.
\end{align*}
Then, the connection $\nabla_{e}$ defines the isomorphism 
\begin{align}
\label{N_e}
N_{e}:\mathbb{Z}E_{i(e)}(\Gamma)\to \mathbb{Z}E_{t(e)}(\Gamma)\in GL(m;\mathbb{Z})
\end{align}
by the inverse (or equivalently the transpose) of the permutation $(m\times m)$-square matrix. More precisely, the square matrix $N_{e}$ is defined as follows.
If we put $\mathbb{Z}E_{i(e)}(\Gamma)=\mathbb{Z}e_{1,p}\oplus\cdots\oplus\mathbb{Z}e_{m,p}$ ($p=i(e)$) and 
$\mathbb{Z}E_{t(e)}(\Gamma)=\mathbb{Z}e_{1,q}\oplus\cdots\oplus\mathbb{Z} e_{m,q}$ ($q=t(e)$), then $\nabla_{e}$ induces the following isomorphism:
\begin{align*}
& k_{1}e_{1,p}\oplus\cdots\oplus k_{m}e_{m,p} \\
& \mapsto k_{1}e_{\sigma(1),q}\oplus\cdots\oplus k_{m}e_{\sigma(m),q}=k_{\sigma^{-1}(1)}e_{1,q}\oplus\cdots\oplus k_{\sigma^{-1}(m)}e_{m,q}.
\end{align*}
Thus $N_{e}:\mathbb{Z}E_{i(e)}(\Gamma)\to \mathbb{Z}E_{t(e)}(\Gamma)$ is defined by the square matrix which induces the following isomorphism
\begin{align}
\label{def_N_e}
N_{e}:
\begin{pmatrix}
k_{1} \\
\vdots \\
k_{m}
\end{pmatrix}
\mapsto 
\begin{pmatrix}
k_{\sigma^{-1}(1)} \\
\vdots \\
k_{\sigma^{-1}(m)}
\end{pmatrix}
\end{align}
where $\sigma$ is the permutation induced from $\nabla_{e}$.
Take an edge $e\in E(\Gamma)$.
Recall that the congruence coefficient $c_{e}(e')$ which is defined by~\eqref{cong-rel} is an integer attached on 
every edge $e'\in E_{i(e)}(\Gamma)$ for the fixed edge $e\in E(\Gamma)$.
Therefore, the $m$-tuple of congruence coefficients on $e$ defines the element in $\mathbb{Z}E_{i(e)}(\Gamma)$ by
\[
c_{e}=(c_{e}(e_{1,i(e)}),\ldots,c_{e}(e_{m,i(e)}))\in \mathbb{Z}e_{1,i(e)}\oplus \cdots\oplus \mathbb{Z}e_{m,i(e)}.
\]
Thus we may define the function 
\begin{align*}
c_{(\Gamma,\alpha,\nabla)}:E(\Gamma)\to \mathbb{Z}^{m}\quad \text{by}\quad c_{(\Gamma,\alpha,\nabla)}(e)=c_{e}.
\end{align*}

Because of the following proposition (see also~\cite{Ta}), we call this function $c_{(\Gamma,\alpha,\nabla)}$ an {\it invariant function} of 
extensions of $(\Gamma,\alpha,\nabla)$:

\begin{proposition}
\label{prop:2-5}
For any extensions $(\Gamma,\widetilde{\alpha},\nabla)$ of $(\Gamma,\alpha,\nabla)$,
the equation 
$c_{(\Gamma,\alpha,\nabla)}=c_{(\Gamma,\widetilde{\alpha},\nabla)}$ holds.
\end{proposition}
\begin{proof}
Let $(\Gamma,\widetilde{\alpha},\nabla)$ be an $(m,\ell)$-type GKM graph for some $\ell>n$, and 
$\widetilde{c}_{e}(e')$ be its congruence coefficient of $e'$ on $e$.
Fix an order of out-going edges on each vertex $p$.
By definition of the function $c_{(\Gamma,\alpha,\nabla)}$,
it is enough to prove that the equation
$c_{e}(e')=\widetilde{c}_{e}(e')$ for all $e\in E(\Gamma)$ and $e'\in E_{i(e)}(\Gamma)$.

By definition, there is a projection $\pi:H^{2}(BT^{\ell})\to H^{2}(BT^{n})$ 
such that $\pi\circ \widetilde{\alpha}=\alpha$.
Together with the congruence relations~\eqref{cong-rel}, we have 
\begin{align*}
\pi(\widetilde{\alpha}(\nabla_{e}(e')))=\alpha(\nabla_{e}(e'))=\alpha(e')+c_{e}(e')\alpha(e)
\end{align*}
and 
\begin{align*}
\pi(\widetilde{\alpha}(\nabla_{e}(e')))&=\pi(\widetilde{\alpha}(e')+\widetilde{c}_{e}(e')\widetilde{\alpha}(e)) \\
&=\pi(\widetilde{\alpha}(e'))+\widetilde{c}_{e}(e')\pi(\widetilde{\alpha}(e)) \\
&=\alpha(e')+\widetilde{c}_{e}(e')\alpha(e).
\end{align*}
Comparing these equations, we establish the statement.
\end{proof}

The following lemma tells us that the $c_{(\Gamma,\alpha,\nabla)}(\overline{e})$ is automatically determined by $c_{(\Gamma,\alpha,\nabla)}(e)$ and $N_{e}$ defined in~\eqref{N_e}.
\begin{lemma}
\label{lem:2-6}
For any $e\in E(\Gamma)$, the equation 
$N_{e}(c_{(\Gamma,\alpha,\nabla)}(e))=c_{(\Gamma,\alpha,\nabla)}(\overline{e})$
holds.
\end{lemma}
\begin{proof}
Let $\sigma$ be the permutation on $\{1,\ldots,m\}$ induced from $\nabla_{e}$.
Then, $\nabla_{e}(e_{j,i(e)})=e_{\sigma(j),i(\overline{e})}$ (and $\nabla_{\overline{e}}(e_{\sigma(j),i(\overline{e})})=e_{j,i(e)}$).
Therefore, 
by definitions of $N_{e}$ and $c_{(\Gamma,\alpha,\nabla)}$, it is enough to show the following equality: $c_{e}(e_{\sigma^{-1}(j),i(e)})=c_{\overline{e}}(e_{j,i(\overline{e})})$, i.e., 
\begin{align*}
& c_{e}(e_{j,i(e)})=c_{\overline{e}}(e_{\sigma(j),i(\overline{e})})
\end{align*}
for all $j=1,\ldots,m$.
Because of the congruence relations~\eqref{cong-rel} on $e$ and $\overline{e}$, we have 
\begin{align*}
\alpha(\nabla_{e}(e_{j,i(e)}))-\alpha(e_{j,i(e)})
& =c_{e}(e_{j,i(e)})\alpha(e) \\
& =\alpha(e_{\sigma(j),i(\overline{e})})-\alpha(e_{j,i(e)})
\end{align*}
and 
\begin{align*}
\alpha(\nabla_{\overline{e}}(e_{\sigma(j),i(\overline{e})}))-\alpha(e_{\sigma(j),i(\overline{e})})
&=c_{\overline{e}}(e_{\sigma(j),i(\overline{e})})\alpha(\overline{e}) \\
&=\alpha(e_{j,i(e)})-\alpha(e_{\sigma(j),i(\overline{e})}).
\end{align*}
By these equations and $\alpha(e)=-\alpha(\overline{e})$, we establish the statement.
\end{proof}

\begin{example}
\label{exam:2-7}
Figure~\ref{eg-invfct} shows an example of the invariant function $c_{(\Gamma,\alpha,\nabla)}$ induced from the $(3,2)$-GKM graph $(\Gamma,\alpha)$.
In this case, we put the order of $E_{p}(\Gamma)$ as in Figure~\ref{eg-invfct} by 
\begin{align*}
e_{1,p}=e_{1},\quad e_{2,p}=e_{2},\quad e_{3,p}=e_{3},
\end{align*} 
and the order of $E_{q}(\Gamma)$ similarly by 
\begin{align*}
e_{1,q}=\overline{e_{1}},\quad e_{2,q}=\overline{e_{2}},\quad e_{3,q}=\overline{e_{3}}.
\end{align*}
Then, by using the congruence relation for the axial function defined in Figure~\ref{eg-invfct}, the connection $\nabla_{e_{i}}$ for $i=1,2,3$ is determined uniquely by $e_{i}\mapsto \overline{e}_{i}$ and $e_{j}\mapsto \overline{e}_{k}$, where $\{i,j,k\}=\{1,2,3\}$ in Figure~\ref{eg-invfct}.
For example, for the edge $e_{1}$, 
by definition of $c_{(\Gamma,\alpha,\nabla)}$ and Lemma~\ref{lem:2-4}, we have  
\[
c_{(\Gamma,\alpha,\nabla)}(e_{1})=(-2, 1, 1).
\]
With the similar computation (and using Lemma~\ref{lem:2-6}), we obtain $c_{(\Gamma,\alpha,\nabla)}$ as 
the right graph in Figure~\ref{eg-invfct}.
\begin{figure}[h]
\begin{center}
\includegraphics[width=300pt,clip]{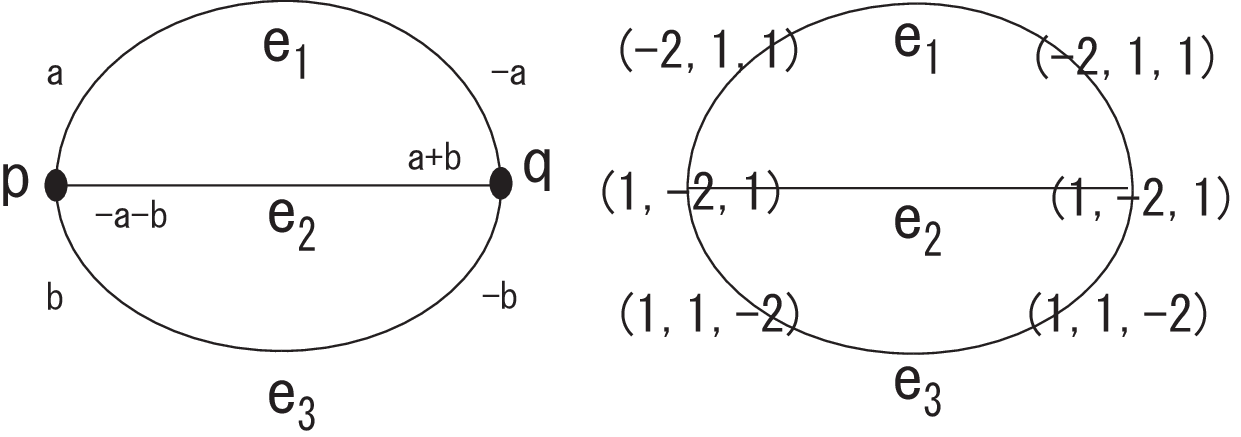}
\caption{The left one is the $(3,2)$-GKM graph $(\Gamma,\alpha)$ in Figure~\ref{examples} and the right one is its invariant function $c_{(\Gamma,\alpha,\nabla)}:E(\Gamma)\to \mathbb{Z}^3$.}
\label{eg-invfct}
\end{center}
\end{figure}
\end{example}

\subsection{A group of axial functions $\mathcal{A}(\Gamma,\alpha,\nabla)$ of $(\Gamma,\alpha,\nabla)$}
\label{sect:2.3}

Let $(\Gamma,\alpha,\nabla)$ be an $(m,n)$-type GKM graph.
In this subsection, we introduce a finitely generated free abelian group $\mathcal{A}(\Gamma,\alpha,\nabla)$, called a {\it group of axial functions}.
Before defining $\mathcal{A}(\Gamma,\alpha,\nabla)$, we prepare some notations.
Let $f\in \mathbb{Z}E_{p}(\Gamma)$.
The symbol $f_{e}(\in \mathbb{Z})$ represents the integer of the coefficient in $f$ corresponding to the edge $e\in E_{p}(\Gamma)$.
For example, if we put the order $E_{p}(\Gamma)=\{e_{1},\ldots,e_{m}\}$ and $f=(x_{1},\ldots,x_{m})\in \mathbb{Z}^{m}\simeq \mathbb{Z}E_{p}(\Gamma)$ with respect to this order,
then we put $f_{e_{j}}=x_{j}$.

\begin{definition}[Group of axial functions]
\label{def-main}
A $\mathbb{Z}$-module $\mathcal{A}(\Gamma,\alpha,\nabla)$ is defined by the submodule of 
\begin{align*}
\{f:V(\Gamma)\to \mathbb{Z}^{m}\}=\bigoplus_{p\in V(\Gamma)}\mathbb{Z}E_{p}(\Gamma)\simeq \bigoplus_{p\in V(\Gamma)}\mathbb{Z}^{m}
\end{align*}
 which satisfies that the following relations for all $e\in E(\Gamma)$:
\begin{align}
\label{def-rel}
N_{e}(f(p))-f(q)=f(q)_{\overline{e}}c_{(\Gamma,\alpha,\nabla)}(\overline{e})
\end{align}
where $i(e)=p$, $t(e)=q$, $N_{e}:\mathbb{Z}E_{p}(\Gamma)\to \mathbb{Z}E_{q}(\Gamma)$ is the square matrix defined in \eqref{def_N_e} and $f(q)_{\overline{e}}\in \mathbb{Z}$ is the integer defined just before.
This module $\mathcal{A}(\Gamma,\alpha,\nabla)$ is said to be a {\it group of axial functions} of $(\Gamma,\alpha,\nabla)$ (also see Remark~\ref{reason-name}).
\end{definition}

\begin{remark}
Because the set $\{f:V(\Gamma)\to \mathbb{Z}^{m}\}$ is a finitely generated free $\mathbb{Z}$-module 
with rank $m|V(\Gamma)|$ i.e., a free abelian group with finite rank, its submodule $\mathcal{A}(\Gamma,\alpha,\nabla)$ is so too.
It is also easy to check that two groups of axial functions with  different orders for $E_{p}(\Gamma)$ (for any $p\in V(\Gamma)$) 
are isomorphic. 
\end{remark}

\begin{remark}
We give a brief remark for our group $\mathcal{A}(\Gamma,\alpha,\nabla)$ from the point of view of the GKM theory.
Let $M$ be a GKM manifold and $(\Gamma_{M},\alpha_{M},\nabla_{M})$ be its induced GKM graph (see Section~\ref{sect:4}).
Recall the famous GKM description of equivariant cohomology of $M$ (see e.g. \cite{GKM, GHZ, GuZa}).
The GKM description of equivariant cohomology of $M$ can be defined by taking the {\it global sections of a sheaf} 
of $(\Gamma_{M},\alpha_{M},\nabla_{M})$ in the sense of Braden-MacPherson~\cite{BrMa} (also see~\cite{Fi}); 
such sheaf is called a {\it structure sheaf} of $(\Gamma_{M},\alpha_{M},\nabla_{M})$ in \cite{Fi} (or {\it sheaf of rings} in \cite{BrMa}). 
The ring defined by $(\Gamma_{M},\alpha_{M},\nabla_{M})$ is often denoted as $H^{*}(\Gamma_{M},\alpha_{M})$ 
and it is isomorphic to the torus equivariant cohomology $H_{T}^{*}(M;\mathbb{Q})$ of $M$ with the rational coefficient, 
if $M$ satisfies the condition called an {\it equivariantly formal} (see \cite{GKM}).
On the other hand, a group of axial functions $\mathcal{A}(\Gamma_{M},\alpha_{M},\nabla_{M})$ can be defined 
by taking a (little bit modified) global sections of another sheaf on $\Gamma$, which is induced from $c_{(\Gamma_{M},\alpha_{M},\nabla_{M})}$.
We omit the details of this process 
because it is apart from the purpose of this paper.
We also finally note that a connection $\nabla_{M}$ is not used to define $H^{*}(\Gamma_{M},\alpha_{M})$; however, 
for the case of $\mathcal{A}(\Gamma_{M},\alpha_{M},\nabla_{M})$, the connection plays the essential role (see Remark~\ref{reason-name}).
\end{remark}

The following corollary follows immediately from Definition~\ref{def-main} and Proposition~\ref{prop:2-5}:
\begin{corollary}
\label{cor:2-11}
Let $(\Gamma,\alpha,\nabla)$ be a GKM graph and $(\Gamma,\widetilde{\alpha},\nabla)$ be its extension.
Then, two groups of axial functions are isomorphic, i.e., $\mathcal{A}(\Gamma,\alpha,\nabla)\simeq \mathcal{A}(\Gamma,\widetilde{\alpha},\nabla)$. 
\end{corollary}

We note the following property of E.q.~\eqref{def-rel} which will be useful to compute $\mathcal{A}(\Gamma,\alpha,\nabla)$.
\begin{lemma}
\label{lem:2-12}
Let $f:V(\Gamma)\to \mathbb{Z}^{m}$ be any function. 
If E.q.~\eqref{def-rel} holds for some edge $e\in E(\Gamma)$, then $f(p)_{e}=-f(q)_{\overline{e}}$ and 
E.q.~\eqref{def-rel} also holds for the edge $\overline{e}$, where $p=i(e)$ and $q=t(e)$.
\end{lemma}
\begin{proof} 
Let $e=e_{j,p}$ and $\overline{e}=e_{\sigma(j),q}$, where $\sigma$ is the permutation induced from $\nabla_{e}$. 
Put
\begin{align*}
f(p)=\begin{pmatrix}
k_{1,p} \\
\vdots \\
k_{m,p}
\end{pmatrix}.
\end{align*}
Then, $f(p)_{e}=k_{j,p}$ and
\begin{align*}
N_{e}(f(p))=
\begin{pmatrix}
x_{1} \\
\vdots \\
x_{m}
\end{pmatrix}
=\begin{pmatrix}
k_{\sigma^{-1}(1),p} \\
\vdots \\
k_{\sigma^{-1}(m),p}
\end{pmatrix}.
\end{align*}
This shows that $N_{e}(f(p))_{\overline{e}}=x_{\sigma(j)}=k_{j,p}=f(p)_{e}.$
Therefore, together with Lemma~\ref{lem:2-4}, we have 
\[
N_{e}(f(p))_{\overline{e}}-f(q)_{\overline{e}}=f(p)_{e}-f(q)_{\overline{e}}=-2f(q)_{\overline{e}}.
\]
So $f(p)_{e}=-f(q)_{\overline{e}}$.
Note that $N_{\overline{e}}=N_{e}^{-1}$ because $\nabla_{\overline{e}}=\nabla_{e}^{-1}$.
Thus, evaluating E.q.~\eqref{def-rel} by $N_{\overline{e}}$ and using Lemm~\ref{lem:2-6}, 
we have 
\begin{align*}
f(p)-N_{\overline{e}}(f(q))=-f(p)_{e}N_{\overline{e}}(c_{(\Gamma,\alpha,\nabla)}(\overline{e}))=-f(p)_{e}c_{(\Gamma,\alpha,\nabla)}(e).
\end{align*}
This establishes the statement.
\end{proof}

\begin{example}
\label{eg-O}
Before proving the main theorem, 
let us compute the $\mathcal{A}(\Gamma,\alpha,\nabla)$ of $(\Gamma,\alpha,\nabla)$ in Example~\ref{exam:2-7}.
By definition of $\mathcal{A}(\Gamma,\alpha,\nabla)$, we first have
\begin{align*}
\mathcal{A}(\Gamma,\alpha,\nabla)&=\{ f:\{p,q\}\to \mathbb{Z}^{3} \mid  
 N_{e_{i}}(f(p))-f(q)=f(q)_{\overline{e}_{i}} c_{(\Gamma,\alpha,\nabla)}(\overline{e}_{i})\}
\end{align*}
Put $f(p)=(x,y,z)\in \mathbb{Z}E_{p}(\Gamma)$ and $f(q)=(x',y',z')\in \mathbb{Z}E_{q}(\Gamma)$.
Then, by Lemma~\ref{lem:2-12}, $x'=-x$, $y'=-y$, $z'=-z$.
Therefore, for example for the case when $i=1$, the relation of $\mathcal{A}(\Gamma,\alpha,\nabla)$ says that 
\begin{align*}
\begin{pmatrix}
1 & 0 & 0 \\
0 & 0 & 1 \\
0 & 1 & 0
\end{pmatrix}
\begin{pmatrix}
x \\
y \\
z
\end{pmatrix}
-
\begin{pmatrix}
-x \\
-y \\
-z
\end{pmatrix}
=
-x
\begin{pmatrix}
-2 \\
1 \\
1
\end{pmatrix}
\end{align*}
Hence, we also have the relation $x+y+z=0$.
Similarly, computing for the other edges $e_{2}, e_{3}$ (Lemma~\ref{lem:3-2} proved later may be also useful), 
we get 
\begin{align*}
\mathcal{A}(\Gamma,\alpha,\nabla)=
\{(f(p),f(q))=((x,y,z),(-x,-y,-z)) \mid x+y+z=0\} (\simeq \mathbb{Z}^{2}).
\end{align*}
\end{example}

\section{Main theorem}
\label{sect:3}

In this section, we prove the following main theorem.
\begin{theorem} 
\label{thm-main}
Let $(\Gamma,\alpha,\nabla)$ be an $(m,n)$-type GKM graph.
Then the following two statements are equivalent:
\begin{enumerate}
\item there is an (m,$\ell$)-type GKM graph which is an extension of $(\Gamma,\alpha,\nabla)$ for some $\ell\ge n$;
\item $\ell \le {\rm rk}~\mathcal{A}(\Gamma,\alpha,\nabla) (\le m)$.
\end{enumerate}
In particular, if ${\rm rk}~\mathcal{A}(\Gamma,\alpha,\nabla)=k$, 
then there is an extended (m,k)-type GKM graph $(\Gamma,\widetilde{\alpha},\nabla)$ which is the maximal among extensions.
\end{theorem}

\subsection{Proof of $(1)\Rightarrow (2)$}
\label{sect:3.1}

We first prove $(1)\Rightarrow (2)$ in Theorem~\ref{thm-main}.
The following lemma is the key lemma:
\begin{lemma}
\label{lem:3-2}
Let $(\Gamma,\alpha,\nabla)$ be an $(m,n)$-type GKM graph.
Then, the rank of $\mathcal{A}(\Gamma,\alpha,\nabla)$ satisfies the following inequality:
\[
n\le {\rm rk}~\mathcal{A}(\Gamma,\alpha,\nabla)\le m.
\]
\end{lemma}
\begin{proof}
We first prove the inequality ${\rm rk}~\mathcal{A}(\Gamma,\alpha,\nabla)\le m$.
By definition, $f\in \mathcal{A}(\Gamma,\alpha,\nabla)\subset \oplus_{p\in V(\Gamma)}\mathbb{Z}E_{p}(\Gamma)$.
Under the same notations in E.q.~\eqref{def-rel}, we put 
\begin{align*}
f(p)=
\begin{pmatrix}
x_{1} \\
\vdots \\
x_{m}
\end{pmatrix},\quad 
f(q)=
\begin{pmatrix}
y_{1} \\
\vdots \\
y_{m}
\end{pmatrix},\quad
c_{(\Gamma,\alpha,\nabla)}(e)=
\begin{pmatrix}
k_{1,p} \\
\vdots \\
k_{m,p}
\end{pmatrix}
\end{align*}
where $x_{j}$, $y_{j}$ are variables and $k_{j,p}\in \mathbb{Z}$ for all $j=1,\ \ldots,\ m$.
Put $f(p)_{e}=x_{j}$.
Then, by E.q.~\eqref{def-rel} for $\overline{e}$, for all $i=1,\ \ldots,\ m$ the following equation holds:
\begin{align*}
y_{\sigma(i)}-x_{i}=x_{j}k_{i,p},
\end{align*}
where $\sigma$ is the permutation on $\{1,\ldots, m\}$ induced from $\nabla_{e}$.
This implies that once we choose the value $f(p)\in \mathbb{Z}^{m}$ for a vertex $p$ which connects with $q$ by an edge, 
then the value $f(q)\in \mathbb{Z}^{m}$ is automatically and uniquely determined.
Because $\Gamma$ is a connected graph, 
iterating this argument on each edge, we can determine the value $f(r)$ uniquely for all $r\in V(\Gamma)$ if we choose a value of $f(p)$.
This implies that the restriction map
\begin{align}
\label{rest-map}
\rho_{p}:\mathcal{A}(\Gamma,\alpha,\nabla)\to \mathbb{Z}^{m}\ \text{such that}\ \rho_{p}(f)=f(p)
\end{align} 
is the injective homomorphism for any vertex $p\in V(\Gamma)$, which proves 
${\rm rk}~\mathcal{A}(\Gamma,\alpha,\nabla)\le m$.

We next prove the rest inequality $n\le {\rm rk}~\mathcal{A}(\Gamma,\alpha,\nabla)$.
Because $(\Gamma,\alpha,\nabla)$ is an $(m,n)$-type GKM graph, 
taking a linear basis of $H^{2}(BT^{n})$ as $\{a_{1},\ldots,a_{n}\}$,
its axial function can be written as $\alpha:E(\Gamma)\to H^{2}(BT^{n})=\mathbb{Z}a_{1}\oplus\cdots \oplus \mathbb{Z}a_{n}\simeq \mathbb{Z}^{n}$.
Put $\pi_{i}:H^{2}(BT^{n})\to \mathbb{Z}a_{i}$ be the projection onto the $i$th coordinate of $H^{2}(BT^{n})$ with respect to this basis.
Define 
\[
\alpha_{i}:E(\Gamma)\stackrel{\alpha}{\longrightarrow} H^{2}(BT^{n})\stackrel{\pi_{i}}{\longrightarrow}\mathbb{Z}a_{i}.
\] 
Recall that we choose an order on $E_{p}(\Gamma)=\{e_{1,p},\ldots,e_{m,p}\}$ for each $p\in V(\Gamma)$.
Put 
\[
\alpha_{i}(e_{j,p})=k_{j,p}^{(i)}a_{i}
\] 
for some $k_{j,p}^{(i)}\in \mathbb{Z}$.
Then, the map $f_{i}:V(\Gamma)\to \mathbb{Z}^{m}$ is defined by 
\begin{align*}
f_{i}(p)=
\begin{pmatrix}
k_{1,p}^{(i)} \\
\vdots \\
k_{m,p}^{(i)}
\end{pmatrix}
\in \mathbb{Z}E_{p}(\Gamma)\simeq \mathbb{Z}^{m},
\end{align*}
for each $p\in V(\Gamma)$.
We claim that $f_{i}\in \mathcal{A}(\Gamma,\alpha,\nabla)$ and $\{f_{1},\ldots, f_{n}\}$ spans the rank $n$ submodule in $\mathcal{A}(\Gamma,\alpha,\nabla)$.
Let $p=i(e),\ q=t(e)\in V(\Gamma)$ for some $e\in E(\Gamma)$.
In order to prove $f_{i}\in \mathcal{A}(\Gamma,\alpha,\nabla)$, by definition, it is enought to show the equation
\begin{align}
\label{eq_lem:3-2}
 N_{e}(f_{i}(p))-f_{i}(q)=f_{i}(q)_{\overline{e}}c_{(\Gamma,\alpha,\nabla)}(\overline{e}).
\end{align}
Now we have
\[
\alpha(\nabla_{e}(e_{j,p}))=\alpha(e_{\sigma(j),q})=\alpha(e_{j,p})+c_{e}(e_{j,p})\alpha(e).
\]
Taking $\pi_{i}$ on these equations, we obtain 
\[
\alpha_{i}(\nabla_{e}(e_{j,p}))=k_{\sigma(j),q}^{(i)}a_{i}=k_{j,p}^{(i)}a_{i}+c_{e}(e_{j,p}) k_{e}^{(i)}a_{i}
\]
for all $j=1,\ldots,m$,
where $k_{e}^{(i)}a_{i}=\alpha_{i}(e)=f_{i}(p)_{e}a_{i}$.
Therefore, we have 
\begin{align*}
\begin{pmatrix}
k_{\sigma(1),q}^{(i)} \\
\vdots \\
k_{\sigma(m),q}^{(i)}
\end{pmatrix}
=
\begin{pmatrix}
k_{1,p}^{(i)} \\
\vdots \\
k_{m,p}^{(i)}
\end{pmatrix}
+k_{e}^{(i)}
\begin{pmatrix}
c_{e}(e_{1,p}) \\
\vdots \\
c_{e}(e_{m,p})
\end{pmatrix}
\end{align*}
Because $N_{\overline{e}}=N_{e}^{-1}$ (where $N_{e}$ is defined by $\nabla_{e}$), we have the following equation from this equation: 
\[
N_{\overline{e}}(f_{i}(q))=f_{i}(p)+f_{i}(p)_{e}c_{(\Gamma,\alpha,\nabla)}(e).
\]
Thus, by multiplying $N_{e}$ and using Lemma \ref{lem:2-6}, we have  that 
\[
N_{e}(f_{i}(p))-f_{i}(q)=-f_{i}(p)_{e}c_{(\Gamma,\alpha,\nabla)}(\overline{e}).
\]
Now, by definition of $\alpha_{i}$, we have that $f_{i}(p)_{e}=-f_{i}(q)_{\overline{e}}$.
Hence, we obtain E.q.~\eqref{eq_lem:3-2} and establish that $f_{i}\in \mathcal{A}(\Gamma,\alpha,\nabla)$ for all $i=1,\ldots,n$.
Now, by definition of the effective axial function, the collection  
$\{ \alpha(e_{p,1}),\ldots, \alpha(e_{p,m})\}$ spans $H^{2}(BT^{n})\simeq \mathbb{Z}^{n}$ for each $p\in V(\Gamma)$.
This implies that for every $p\in V(\Gamma)$ there is a collection  
$\{f_{1}(p),\ldots,f_{n}(p)\}$ which spans some $n$-dimensional subspace in $\mathbb{Z}E_{p}(\Gamma)(\simeq \mathbb{Z}^{m})$.
Because the restriction map~\eqref{rest-map} is injective, this also implies that the set of functions 
$\{f_{1},\ldots,f_{n}\}$ spans some $n$-dimensional submodule in $\mathcal{A}(\Gamma,\alpha,\nabla)(\subset \mathbb{Z}^{m})$.
This establishes that $n\le {\rm rk}~\mathcal{A}(\Gamma,\alpha,\nabla)$.
\end{proof}

Assume that there is an $(m,\ell)$-type GKM graph $(\Gamma,\widetilde{\alpha},\nabla)$ which is an extension of $(\Gamma,\alpha,\nabla)$ for some $n\le \ell \le m$.
Then, it easily follows from Corollary~\ref{cor:2-11} and Lemma~\ref{lem:3-2} that 
\begin{align*}
\ell\le {\rm rk}~\mathcal{A}(\Gamma,\widetilde{\alpha},\nabla)={\rm rk}~\mathcal{A}(\Gamma,\alpha,\nabla)\le m.
\end{align*}
This establishes the statement $(1)\Rightarrow (2)$ in Theorem~\ref{thm-main}.

\subsection{Proof of $(2)\Rightarrow (1)$}
\label{sect:3.2}
We next prove $(2)\Rightarrow (1)$ in Theorem~\ref{thm-main} for an $(m,n)$-type GKM graph $(\Gamma,\alpha,\nabla)$.
Assume that
\[
\ell\le {\rm rank}\ \mathcal{A}(\Gamma,\alpha,\nabla) (\le m)
\] 
for some $\ell\ge n$.
We shall prove that there exists an extension $(m,\ell)$-type GKM graph $(\Gamma,\widetilde{\alpha},\nabla)$ of the $(m,n)$-type GKM graph $(\Gamma,\alpha,\nabla)$.

Let $f\in \mathcal{A}(\Gamma,\alpha,\nabla)$.
Put the order of $E_{p}(\Gamma)$ as $\{e_{1,p},\ldots,e_{m,p}\}$ for $p\in V(\Gamma)$ and 
\begin{align*}
f(p)=
\begin{pmatrix}
k_{1,p} \\
\vdots \\
k_{m,p}
\end{pmatrix}
\end{align*}
with respect to this order.
Then, we may define $\alpha_{f}^{a}:E(\Gamma)\to \Z a$ for every $a\in H^{2}(BT^{n})$ by 
\begin{align*}
\alpha_{f}^{a}(e_{j,p})=k_{j,p}a.
\end{align*}
We call this label $\alpha_{f}^{a}$ on edges an {\it $a$-labeling induced from $f\in \mathcal{A}(\Gamma,\alpha,\nabla)$}.
In the proof of Lemma~\ref{lem:3-2}, it is easy to see that $\alpha_{i}:=\pi_{i}\circ\alpha=\alpha_{f_{i}}^{a_{i}}$ for $i=1,\ldots,n$.
Therefore, we obtain the following corollary from the proof of Lemma~\ref{lem:3-2}, which is the key fact to prove $(2)\Rightarrow (1)$ and 
tells us that the axial function $\alpha$ can be recovered from $\mathcal{A}(\Gamma,\alpha,\nabla)$ by using $\alpha_{f}^{a}$.
\begin{corollary}
\label{cor:3-3}
Let $(\Gamma,\alpha,\nabla)$ be an $(m,n)$-type GKM graph.
Then, there exsits $f_{i}\in \mathcal{A}(\Gamma,\alpha,\nabla)$ for $i=1,\ldots,n$ 
such that $\{f_{1},\ldots,f_{n}\}$ spans an $n$-dimensional subspace of $\mathcal{A}(\Gamma,\alpha,\nabla)$
and 
for the fixed basis $a_{1},\ldots,a_{n}$ of $H^{2}(BT^{n})$
the axial function can be split into
\[
\alpha_{1}\oplus \cdots \oplus \alpha_{n}=\alpha:E(\Gamma)\to H^{2}(BT^{n})=\bigoplus_{i=1}^{n}\Z a_{i}
\]
where $\alpha_{i}:=\alpha_{f_{i}}^{a_{i}}$ is the $a_{i}$-labeling induced from $f_{i}$.
\end{corollary}

Because $\ell\le {\rm rk}\ \mathcal{A}(\Gamma,\alpha,\nabla)$, there are independent elements (as free $\mathbb{Z}$-module)
\[
f_{1},\ldots,f_{\ell}\in \mathcal{A}(\Gamma,\alpha,\nabla).
\]
Moreover, because of $\ell\ge n$,
we may choose the first part $f_{1},\ldots,f_{n}$ as the basis induced from $(\Gamma,\alpha,\nabla)$ in Corollary~\ref{cor:3-3} (also see the definitions of $f_{i}$'s in the proof of Lemma~\ref{lem:3-2}), and put
\begin{align}
\label{axial-splits-coeff}
f_{i}(p)=
\begin{pmatrix}
k_{1,p}^{(i)} \\
\vdots \\
k_{m,p}^{(i)}
\end{pmatrix}
\in \mathbb{Z}E_{p}(\Gamma)\simeq \mathbb{Z}^{m},
\end{align}
for $i=1,\ldots, \ell$ and $p\in V(\Gamma)$.
Fix the basis of $H^{2}(BT^{\ell})$ as $a_{1},\ldots,a_{\ell}$, where the first $n$ elements $a_{1},\ldots,a_{n}$ are the basis of $H^{2}(BT^{n})$ (see Corollary~\ref{cor:3-3}), where $T^{n}\subset T^{\ell}$.
Let $\alpha_{i}$ be the $a_{i}$-labeling induced from $f_{i}$, i.e., $\alpha_{i}=\alpha_{f_{i}}^{a_{i}}$, for $i=1,\ldots,\ell$.
Then, we can define the function as follows: 
\[
\widetilde{\alpha}:=\oplus_{i=1}^{\ell}\alpha_{i}:E(\Gamma)\to H^{2}(BT^{\ell}).
\]
The following lemma says that the triple $(\Gamma,\widetilde{\alpha},\nabla)$ is a GKM graph extending $(\Gamma,\alpha,\nabla)$.
\begin{lemma}
\label{lem:3-4}
The triple $(\Gamma,\widetilde{\alpha},\nabla)$ defined as above is an $(m,\ell)$-type GKM graph which is an extension of $(m,n)$-type GKM graph $(\Gamma,\alpha,\nabla)$.
\end{lemma}
\begin{proof}
Because 
$\alpha=\oplus_{i=1}^{n}\alpha_{i}$ and $\widetilde{\alpha}=\alpha\oplus (\oplus_{i=n+1}^{\ell}\alpha_{i})$, it is enough to prove that $\widetilde{\alpha}$ satisfies the axiom of GKM graph under the same connection $\nabla$ of $(\Gamma,\alpha,\nabla)$.

We first claim that the axiom (1) of GKM graph holds for $\widetilde{\alpha}$, i.e., $\widetilde{\alpha}(e)=-\widetilde{\alpha}(\overline{e})$.
To do this, 
by definition of $\widetilde{\alpha}$, it is enought to show that 
$\alpha_{i}(e)=-\alpha_{i}(\overline{e})$ for all $i=1,\ldots,\ell$.
Let $e=e_{j,p}\in E_{p}(\Gamma)$ and $\overline{e}=e_{\sigma(j),q}\in E_{q}(\Gamma)$ for $j=1,\ldots,m$, where $i(e)=p$, $t(e)=q$ and $\sigma$ is the permutation on $\{1,\ldots,m\}$ induced from $\nabla_{e}:E_{q}(\Gamma)\to E_{p}(\Gamma)$.
Then, by definition of $\alpha_{i}$, we have that 
\begin{align*}
\alpha_{i}(e)=\alpha_{i}(e_{j,p})=k_{j,p}^{(i)}a_{i}
\end{align*}
and 
\begin{align*}
\alpha_{i}(\overline{e})=\alpha_{i}(e_{\sigma(j),q})=k_{\sigma(j),q}^{(i)}a_{i}
\end{align*}
where $f_{i}(p)_{e}=k_{j,p}^{(i)},\ f_{i}(q)_{\overline{e}}=k_{\sigma(j),q}^{(i)}\in \mathbb{Z}$.
Because $f_{i}\in \mathcal{A}(\Gamma,\alpha,\nabla)$, it follows from Lemma \ref{lem:2-12} that 
$k_{j,p}^{(i)}=-k_{\sigma(j),q}^{(i)}$.
This establish the axiom (1) of GKM graph.

We next claim the condition (4) of effectiveness, i.e., $\widetilde{\alpha}_{(p)}=\{\widetilde{\alpha}(e)\ |\ e\in E_{p}(\Gamma)\}$ spans $H^{2}(BT^{\ell})$ for all $p\in V(\Gamma)$.
Recall that for $E_{p}(\Gamma):=\{e_{1,p},\ldots,e_{m,p}\}$,
\begin{align*}
\widetilde{\alpha}(e_{j,p})=\oplus_{i=1}^{\ell}\alpha_{i}(e_{j,p})=\oplus_{i=1}^{\ell}k_{j,p}^{(i)}a_{i},
\end{align*}
where the integer $k_{j,p}^{(i)}$ is the $j$th coefficient of $f_{i}(p)\in \mathbb{Z}^{m}$ (see \eqref{axial-splits-coeff}).
Now $\{f_{1},\ldots,f_{\ell}\}$ spans an $\ell$-dimensional subspace of $\mathcal{A}(\Gamma,\alpha,\nabla)$.
Because the restriction map defined in \eqref{rest-map} is injective (by the similar arguments in the proof of Lemma~\ref{lem:3-2}), 
we have that the subset $\{f_{1}(p),\ldots,f_{\ell}(p)\}\subset \mathbb{Z}^{m}$ also spans a subgroup which is isomorphic to $\mathbb{Z}^{\ell}$ for all $p\in V(\Gamma)$.
This shows that the $(m\times \ell)$-matrix $(k_{j,p}^{(i)})_{i,j}$ has rank $\ell(\le m)$ 
and some minor (of $(\ell\times \ell)$-smaller square matrix in $(k_{j,p}^{(i)})_{i,j}$) with determinant $\pm 1$, for all $p\in V(\Gamma)$.
Therefore, there are $\ell$ elements in $\widetilde{\alpha}_{(p)}=\{\widetilde{\alpha}(e_{1,p}),\ldots, \widetilde{\alpha}(e_{m,p})\}$ 
which generate $H^{2}(BT^{\ell})$.
This establishes the condition (4).

We also check the axiom (2), i.e., $\widetilde{\alpha}_{(p)}$ is pairwise lineraly independent for all 
$p\in V(\Gamma)$.
Because $\alpha$ is an axial function, $\alpha_{(p)}$ is pairwise linearly independent for all $p\in V(\Gamma)$, i.e.,
$\alpha(e)$ and $\alpha(e')$ are linearly independent for all pairs $e, e'\in E_{p}(\Gamma)$.
Moreover, we may write 
\[
\widetilde{\alpha}(e)=\oplus_{i=1}^{\ell}\alpha_{i}(e)= \oplus_{i=1}^{n}\alpha_{i}(e)\oplus(\oplus_{i=n+1}^{\ell}\alpha_{i}(e))=\alpha(e)\oplus(\oplus_{i=n+1}^{\ell}\alpha_{i}(e)),
\]
and 
\[
\widetilde{\alpha}(e')=\oplus_{i=1}^{\ell}\alpha_{i}(e')= \oplus_{i=1}^{n}\alpha_{i}(e')\oplus(\oplus_{i=n+1}^{\ell}\alpha_{i}(e'))=\alpha(e')\oplus(\oplus_{i=n+1}^{\ell}\alpha_{i}(e')).
\]
Here, by definition of $\alpha_{i}$, the element $\alpha_{i}(e)$ (resp.\ $\alpha_{i}(e')$), for $i=n+1,\ldots, \ell$, is independent with 
$\alpha(e)$ (resp.\ $\alpha(e')$). 
Hence, we have that $\widetilde{\alpha}(e)$ and $\widetilde{\alpha}(e')$ are also pairwise linearly independent. 
This establishes the axiom (2).

Finally, we claim the axiom (3), i.e., $\widetilde{\alpha}$ satisfies the following congruence relation:
for each $e'\in E_{i(e)}(\Gamma)$ 
\[
\widetilde{\alpha}(\nabla_{e}(e'))=\widetilde{\alpha}(e')+c_{e}(e')\widetilde{\alpha}(e),
\]
where $c_{e}(e')$ is the integer which satisfies that 
\[
\alpha(\nabla_{e}(e'))=\alpha(e')+c_{e}(e')\alpha(e).
\]
Because $\widetilde{\alpha}=\oplus_{i=1}^{\ell}\alpha_{i}$, it is enough to prove that $\alpha_{i}$ satisfies the congruence relation:
\[
\alpha_{i}(\nabla_{e}(e'))=\alpha_{i}(e')+c_{e}(e')\alpha_{i}(e).
\]
Set $e=e_{j,p}$ and $e'=e_{h,p}$ for some $j,h=1,\ldots, m$, and $\nabla_{e}(e')=e_{\sigma(h),q}$.
By definition, $\alpha_{i}(e_{j,p})=k_{j,p}^{(i)}a_{i}$.
Therefore, it is enough to check the following relation:
\begin{align}
\label{rel-goal}
k_{\sigma(h),q}^{(i)}=
k_{h,p}^{(i)}+c_{e}(e')k_{j,p}^{(i)}.
\end{align}
Using $f_{i}\in \mathcal{A}(\Gamma,\alpha,\nabla)$, (i.e., E.q.~\eqref{def-rel} holds), Lemma~\ref{lem:2-6} and Lemma~\ref{lem:2-12},
we have 
\begin{align}
\label{relation-prove-axiom3}
N_{e}(f_{i}(p))-f_{i}(q)=-f_{i}(p)_{e} N_{e}(c_{(\Gamma,\alpha,\nabla)}(e))
\end{align}
for all $i=1,\ldots, \ell$.
Since $e=e_{j,p}$, we have $f_{i}(p)_{e}=k_{j,p}^{(i)}$ (see \eqref{axial-splits-coeff}).
Therefore,   
E.q.~\eqref{relation-prove-axiom3} implies that  
\begin{align*}
\begin{pmatrix}
k_{\sigma^{-1}(1),p}^{(i)} \\
\vdots \\
k_{\sigma^{-1}(m),p}^{(i)}
\end{pmatrix}
-
\begin{pmatrix}
k_{1,q}^{(i)} \\
\vdots \\
k_{m,q}^{(i)}
\end{pmatrix}
=-k_{j,p}^{(i)}
\begin{pmatrix}
c_{e}(e_{\sigma^{-1}(1),p}) \\
\vdots \\
c_{e}(e_{\sigma^{-1}(m),p})
\end{pmatrix}.
\end{align*}
This shows that 
\begin{align*}
k_{h,p}^{(i)}-k_{\sigma(h),q}^{(i)}=-k_{j,p}^{(i)}c_{e}(e_{h,p}).
\end{align*}
Since $e'=e_{h,p}$, 
this equation establishes the equation \eqref{rel-goal}.

Consequently, $\widetilde{\alpha}$ is an extended axial function of $\alpha$.
\end{proof}

This establishes $(2)\Rightarrow (1)$ in Theorem~\ref{thm-main}.
Together with Section~\ref{sect:3.1}, we obtain Theorem \ref{thm-main}.

\begin{remark}
\label{reason-name}
Lemma~\ref{lem:3-4} tells us that 
from an element of $\mathcal{A}(\Gamma,\alpha,\nabla)$ 
we can construct an extension of $(\Gamma,\alpha,\nabla)$.
In fact, by the similar arguments, we see that  
$\mathcal{A}(\Gamma,\alpha,\nabla)$ contains every extension of $(\Gamma,\alpha,\nabla)$, 
i.e., every extension of $(\Gamma,\alpha,\nabla)$ corresponds to an element of $\mathcal{A}(\Gamma,\alpha,\nabla)$.
Furthermore, it is not so difficult to show that every axial function on $\Gamma$ whose connection is $\nabla$ can be constructed by an element of $\mathcal{A}(\Gamma,\alpha,\nabla)$.
This is the reason why we call $\mathcal{A}(\Gamma,\alpha,\nabla)$ a {\it group of axial functions}. 
\end{remark}

As a corollary, we have the following fact.
\begin{corollary}
\label{cor:3-6}
Let $(\Gamma,\alpha,\nabla)$ be an $(m,n)$-type GKM graph.
If one of the following cases hold, then there are no extensions of $(\Gamma,\alpha,\nabla)$:
\begin{enumerate}
\item $m=n$;
\item ${\rm rk}\ \mathcal{A}(\Gamma,\alpha,\nabla)=n$.
\end{enumerate}
\end{corollary}

\begin{example}
\label{exam:3-7}
By Corollary~\ref{cor:3-6} and the computation in Example~\ref{eg-O}, the GKM graphs in Figure~\ref{examples} have no extensions, i.e., they are the maximal GKM graphs.
\end{example}

\section{Applications to geometry}
\label{sect:4}

Guillemin-Zara study GKM graph as a combinatorial counter part of the {\it GKM manifold}, and 
they build a bridge between the geometry of (in particular, symplectic) GKM manifolds and the combinatorics of GKM graphs.
In this and next sections, we give a new application of GKM graphs to study the geometry of GKM manifolds.
More precisely, we apply our main result to study the maximal torus actions of GKM manifolds.

\subsection{GKM manifold and its GKM graph}
\label{sect:4.1}

We first briefely recall the relation between GKM manifolds and GKM graphs (see \cite{GuZa} for details).
Let $M$ be a $2m$-dimensional, compact, connected smooth manifold with an effective $n$-dimensional torus $T^{n}$-action, 
where $1\le n\le m$.
We often denote such manifold as 
$(M, T)$ or $(M, T,\varphi)$ if we emphasize the action $\varphi:T\times M\to M$.
Put $M_{1}\subset M$ by the set of elements $x\in M$
such that the orbit $T(x)=\{x\}$ (a fixed point) or $T(x)\simeq S^{1}$, i.e., 
\[
M_{1}=\{x\in M\ |\ \dim T(x)\le 1\}.
\]
The set $M_{1}$ is called a {\it one-skeleton} of $(M,T)$.

A $2m$-dimensional manifold with an $n$-dimensional torus action $(M,T)$ is said to be a {\it GKM manifold} or an {\it $(m,n)$-type GKM manifold} if the following three conditions hold: 
\begin{enumerate}
\item $M^{T}\not=\emptyset$;
\item the manifold $M$ has a $T$-invariant almost complex sturcture; 
\item the one-skeleton of $M$ has the structure of an abstract (connected) graph $\Gamma_{M}$ such that 
its vertices $V(\Gamma_{M})$ are the fixed points and its edges $E(\Gamma_{M})$ are embedded $2$-spheres connecting two fixed points.
\end{enumerate}

\begin{remark}
The third condition implies that the orbit space of the one-skeleton is one-dimensional.
Therefore, by definition of the $(m,n)$-type GKM manifold $M$, 
if $\dim T(=n)=1$ then $M$ is equivariantly diffeomorphic to $\C P^{1}$ with a non-trivial $S^{1}$-action.
So, in this paper, we often assume $2\le n\le m$ for an $(m,n)$-type GKM manifold.
%
\end{remark}

By using the differentialble slice theorem, it is easy to check that $\Gamma_{M}$ is an $m$-valent graph.
An axial function $\alpha_{M}:E(\Gamma_{M})\to H^{2}(BT)$ can be defined as the following way.
Now the cohomology $H^{2}(BT)\simeq \mathbb{Z}^{n}$ may be regarded as 
the set of all complex one-dimensional representations, or equivalently 
the set of all homomorphisms from $T^{n}$ to $S^{1}$, say ${\rm Hom}(T^{n}, S^{1})$. 
Because there is a $T$-invariant complex structure on $M$,
its tangent space $T_{p}M$ is a complex $T$-representation space, 
called a {\it tangential representaion} of $M$ on $p\in M^{T}$.
Therefore, $T_{p}M$ decomposes into irreducible complex representation spaces: 
\[
T_{p}M=\oplus_{i=1}^{m}V(a_{i})
\]
where $a_{i}\in {\rm Hom}(T^{n},S^{1})$ and 
$V(a_{i})$ represents the complex one-dimensional representation space with the weight  
$a_{i}\in {\rm Hom}(T,S^{1})$.
Since the $T$-action is effective, $\{a_{1},\ldots,a_{m}\}$ spans ${\rm Hom}(T^{n},S^{1})\simeq \mathbb{Z}^{n}$.
Moreover, as is well-known,  
the third condition of the definition of GKM manifold is equivalent to that $\{a_{1},\ldots,a_{m}\}$ is pairwise linearly independent and the one-skeleton of $(M,T)$ is connected.
Then, each $V(a_{i})$ may be regarded as the tangential representation of some $T$-invariant $2$-sphere on $p\in M^{T}$, say $S^{2}_{i,p}\subset M$. 
Recall that $E(\Gamma_{M})$ is the set of $T$-invariant $2$-spheres; 
therefore, there is the corresponding edge $e_{i,p}\in E_{p}(\Gamma_{M})$ with $S^{2}_{i,p}$.
This defines the function such that $e_{i,p}\mapsto a_{i}$, and 
we denote this function as $\alpha_{M}:E(\Gamma_{M})\to H^{2}(BT)$.
We next check that the function $\alpha_{M}$ satisfies the axiom of the axial function.
The 1st axiom, i.e., $\alpha_{M}(e)=-\alpha_{M}(\bar{e})$, can be easily checked from  
the fact that the invariant $2$-sphere is isomorphic to the standard $S^{1}$-action on $\mathbb{C} P^{1}(\simeq S^{2})$.
The 2nd axiom and the 4th condition, i.e., the subset $\{\alpha_{M}(e)\ |\ E_{p}(\Gamma_{M})\}$ in $H^{2}(BT)$ for each $p\in V(\Gamma_{M})$ satisfies the pairwise linearly independent and spans $H^{2}(BT)$, 
have already checked in the arguments as above.
In order to check the 3rd axiom, we need to define the connection on the labeled graph $(\Gamma_{M},\alpha_{M})$.
Denote an invariant two sphere which connecting two fixed points $p,q\in M^{T}$ as $S^{2}(pq)\simeq \mathbb{C} P^{1}$ (this might not be unique).
Then, by using the restricted $T^{n}$-invariant almost complex structure on the restricted tangent bundle $\tau_{M}|_{S^{2}(pq)}$,
the restricted tangent bundle $\tau_{M}|_{S^{2}(pq)}$ splits into the following $T^{n}$-invariant line bundles: 
\[
\mathbb{L}_{1}\oplus \cdots \oplus \mathbb{L}_{m},
\]
where $\mathbb{L}_{i}$ is a complex $T^{n}$-equivariant line bundle over $S^{2}(pq)$ for all $i=1,\ldots,m$.
Because each $\mathbb{L}_{i}$ is $T^{n}$-equivariant, we may write the restrictions onto fixed points as $\mathbb{L}_{i}|_{p}\simeq V(a_{i,p})$ and $\mathbb{L}_{i}|_{q}\simeq V(a_{i,q})$
for some $\alpha_{M}(e_{i,p})=a_{i,p}$ and $\alpha_{M}(e_{i,q})=a_{i,q}$, for each $i=1,\ldots, m$.
This defines the map $e_{i,p}\mapsto e_{i,q}$ for all $i=1,\ldots,m$. 
Therefore, there is the bijection $\nabla_{pq}:E_{p}(\Gamma_{M})\to E_{q}(\Gamma_{M})$.
The collection of the bijections on each $e\in E(\Gamma_{M})$ define the collection $\nabla_{M}=\{\nabla_{e}\ |\ e\in E(\Gamma_{M}) \}$.
It is easy to check that this satisfies $\nabla_{e}(e)=\overline{e}$ and $\nabla_{\overline{e}}=\nabla_{e}^{-1}$.
Now it is well-known that for every complex line bundle $\mathbb{L}$ over $\mathbb{C} P^{1}$ there exists an $S^{1}$ representation space
$\mathbb{C}_{r}$ by $r\in {\rm  Hom}(S^{1},S^{1})\simeq \mathbb{Z}$ such that 
\begin{eqnarray}
\label{line-bdl-gen}
\mathbb{L}\equiv\mathbb{L}_{\rho}:=S^{3}\times_{S^{1}}\mathbb{C}_{r},
\end{eqnarray}
where $S^{1}$ acts on $S^{3}\subset \mathbb{C}^{2}$ diagonally and $\mathbb{C}_{r}$ via $r$.
Therefore, it easily follows from \eqref{line-bdl-gen} that $\nabla_{M}$ satisfies the congruence relation \eqref{cong-rel}.
Hence, $\nabla_{M}$ is the connection on $(\Gamma_{M},\alpha_{M})$, and $\alpha_{M}$ is the axial function on $\Gamma_{M}$.

Consequently, 
the GKM manifold $M$ defines the GKM graph $(\Gamma_{M},\alpha_{M},\nabla_{M})$  
by using its one-skeleton and the tangential representations.
In this paper, such GKM graph $(\Gamma_{M},\alpha_{M})$ (or $(\Gamma_{M},\alpha_{M},\nabla_{M})$) is called an {\it induced GKM graph} from $M$.

\begin{example}
\label{ex:4.2}
In Figure \ref{examples}, the left GKM graph is the GKM graph induced from the standard $T^{2}$-action on $\mathbb{C} P^{2}$ and the 
right one is that induced from $T^{3}$-action on $\mathbb{C} P^{3}$.
The middle GKM graph is induced from the $T^{2}$-action on $S^{6}=G_{2}/SU(3)$, 
where $G_{2}$ is the exceptional Lie group, see \cite[Section 5.2]{GHZ}.
\end{example}

\subsection{Extensions of torus actions}
\label{sect:4.2}

The definition of an extension of GKM graphs in Section \ref{sect:2.1} is motivated 
by an extension of a torus action on GKM manifold.
We explain it more preceisely in this section.

Let $(M,T^{n},\varphi)$ be a manifold with an effective $n$-dimensional torus action $\varphi:T^{n}\times M\to M$ (not necessarily a GKM manifold). 
If there exists an effective $\ell$-dimensional torus action $(M,T^{\ell}, \varphi')$ (for $n<\ell$) and 
an injective homomorphism $\iota:T^{n}\to T^{\ell}$ such that the following diagram commutes:
\begin{align*}
\xymatrix{
T^{\ell}\times M \ar[rd]^{\varphi'} &    \\
T^{n}\times M \ar[u]^{\iota \times id} \ar[r]^{\varphi} & M  
}
\end{align*} 
then $(M,T^{\ell}, \varphi')$ is called an {\it extension} of $(M,T^{n},\varphi)$.
We prove the following fact:
\begin{proposition}
\label{ext-cat-in-GKM}
If $(M^{2m},T^{\ell})$ is an extension of an $(m,n)$-type GKM manifold $(M^{2m},T^{n})$ (for $n<\ell\le m$) and 
the $T^{\ell}$-action preserves the almost complex structure of $M$, 
then 
$(M^{2m},T^{\ell})$ is an $(m,\ell)$-type GKM manifold.

Furthermore, the induced $(m,\ell)$-type GKM graph 
$(\widetilde{\Gamma}_{M},\widetilde{\alpha}_{M},\widetilde{\nabla}_{M})$ from $(M,T^{\ell})$ is an extension of the induced $(m,n)$-type GKM graph
$(\Gamma_{M},\alpha_{M},\nabla_{M})$ from $(M,T^{n})$.
\end{proposition}
\begin{proof}
Because the $T^{\ell}$-action preserves the almost complex structure of $M$, 
it is enough to check that its one-skeleton has the structure of a graph.
We note that for all $p\in M$ two orbits of $p$ of these actions satisfy $T^{n}(p)\subset T^{\ell}(p)$,
because the $T^{\ell}$-action is an extension of $T^{n}$-action.

We first claim that $M^{T^{n}}=M^{T^{\ell}}$.
Because $T^{n}(p)\subset T^{\ell}(p)$ for all $p\in M$,
we have that $M^{T^{n}}\supset M^{T^{\ell}}$.
Assume that there exists a fixed point $p\in M^{T^{n}}$ such that $T^{\ell}(p)\not=\{p\}$.
As is well-known, there is a decomposition $T_{p}M=T_{p}T^{\ell}(p)\oplus N_{p}T^{\ell}(p)$, where $T_{p}T^{\ell}(p)(\not=\{0\})$ is the tangent space and $N_{p}T^{\ell}(p)$ is the normal space of $T^{\ell}(p)$ on $p$.
By using the differentiable slice theorem, the isotropy subgroup $T^{\ell}_{p}$ (of the $T^{\ell}$-action on $p$) acts on $T_{p}T^{\ell}(p)$ trivially.
This shows that the $T^{n}(\subset T^{\ell}_{p})$ also acts on $T_{p}T^{\ell}(p)$ trivially.
However, by the definition of GKM manifolds, for the restricted action $(T_{p}M,T^{n})$, 
there is another decomposition $T_{p}M=\oplus_{i=1}^{m}V(a_{i})$
such that each representation $a_{i}:T^{n}\to S^{1}$ is non-trivial.
This contradicts to that $T^{n}$ acts on $(\{0\}\not=)T_{p}T^{\ell}(p)\subset T_{p}M$ trivially.
Hence, $M^{T^{n}}=M^{T^{\ell}}$.

Take $p\in M$ such that $T^{n}(p)\simeq S^{1}$.
Because we assume that the one-skeleton of $(M,T^{n})$ has the structure of a connected graph,
we have that $p$ is an element in an invariant $2$-sphere $S^{2}$ of $(M,T^{n})$.
Because $M^{T^{n}}=M^{T^{\ell}}$, 
by considering the tangential representation around fixed points on this $T^{n}$-invariant $S^{2}(\ni p)$,
there exists a representation $\rho:T^{\ell}\to S^{1}$ which may be regarded as the extension of the $T^{n}$-action on this $S^{2}$.
Therefore, every $T^{n}$-invariant $S^{2}$ is also a $T^{\ell}$-invariant $S^{2}$, i.e., 
if $T^{n}(p)\simeq S^{1}$ then $T^{\ell}(p)\simeq S^{1}$.
Together with $T^{n}(p)\subset T^{\ell}(p)$ for all $p\in M$,
this implies that 
two one-skeletons of $(M^{2m},T^{n})$ and $(M^{2m},T^{\ell})$ are the same.

We next prove the final statement.
By the arguments as above, we have $\widetilde{\Gamma}_{M}=\Gamma_{M}$.
Moreover, because the extended  $T^{\ell}$-action preserves the $T^{n}$-invariant almost complex structure, 
the splitting $\oplus_{i=1}^{m}\LL_{i}$, of the restriction of the tangent bundle to such $S^{2}$, by $T^{n}$-action is also preserved by the extended $T^{\ell}$-action.
This implies that two connections on induced GKM graphs $\nabla_{M}$ from the $T^{n}$-action and $\widetilde{\nabla}_{M}$ from the extended $T^{\ell}$-action are the same.
Finally, put the induced homomorphism from the inclusion $\iota:T^{n}\to T^{\ell}$ as $\pi:H^{2}(BT^{\ell})\to H^{2}(BT^{n})$.
Then, by considering the tangential representations (of both $T^{n}$ and $T^{\ell}$-actions) around fixed points, 
it is easy to check that there is the following commutative diagram:
\begin{align*}
\xymatrix{
 & H^{2}(BT^{\ell})\ar[d]^{\pi}    \\
E(\Gamma_{M}) \ar[ru]^{\widetilde{\alpha}_{M}} \ar[r]^{\alpha_{M}} & H^{2}(BT^{n})  
}
\end{align*} 
This establishes the final statement.
\end{proof}

Therefore, by using Theorem~\ref{thm-main} (or Theorem~\ref{main:1}) and Proposition \ref{ext-cat-in-GKM}, we have Corollary~\ref{main:2}.

\subsection{Maximal torus action on $S^{6}=G_{2}/SU(3)$}
\label{sect:4.3}

As we mentioned in Example~\ref{ex:4.2}, the $(3,2)$-type GKM graph in Figure~\ref{examples} is the induced GKM graph of the GKM manifold $(G_{2}/SU(3),T^{2})$,
where $T^{2}$ acts on $G_{2}$ as its maximal torus subgroup (e.g. see~\cite{GHZ}).
We also note that $G_{2}/SU(3)\simeq S^{6}$ (diffeomorphic).
Therefore, by using Corollary~\ref{main:2} and Example~\ref{eg-O}, the following well-known fact can be proved (see also~\cite{BT}):
\begin{corollary}
\label{ext-S^6}
The $T^{2}$-action on $G_{2}/SU(3)\simeq S^{6}$ is the maximal torus action.
In other words, there are no extended $T^{3}$-actions on $S^{6}$ of this $T^{2}$-action, 
which preserves the almost complex structure induced from the homogeneous space $G_{2}/SU(3)$.
\end{corollary}

\begin{remark}
Note that there is the $T^{3}$-action on $S^{6}\subset \mathbb{C}^{3}\oplus \mathbb{R}$ defined by the standard $T^{3}$-action on $\mathbb{C}^{3}$ 
(see e.g.~\cite{Ku15}).
However, from Corollary~\ref{ext-S^6}, this action is not the extended action of the $T^{2}$-action on $S^{6}=G_{2}/SU(3)$.

\end{remark}

In the next section, we shall apply our results for more complicated GKM manifolds.

\section{Maximal torus action on the complex Grassmannian $G_{2}(\mathbb{C}^{n+2})$}
\label{sect:5}

The {\it (complex) Grassmannian} (of $2$-planes in $\mathbb{C}^{n+2}$), denoted by $G_{2}(\mathbb{C}^{n+2})$, is defined by the set of all complex $2$-dimensional vector spaces in $\mathbb{C}^{n+2}$.
Namely,
\begin{align}
\label{Grass}
G_{2}(\mathbb{C}^{n+2}):=\{V\subset \mathbb{C}^{n+2}\ |\ \dim_{\mathbb{C}} V=2\}.
\end{align}
The Grassmannian $G_{2}(\mathbb{C}^{n+2})$ has the natural transitive $SU(n+2)$-action which is 
induced from the standard $SU(n+2)$-action on $\mathbb{C}^{n+2}$.
Since its isotropy group is $S(U(2)\times U(n))$, $G_{2}(\mathbb{C}^{n+2})$ is diffeomorphic to the homogeneous space $SU(n+2)/S(U(2)\times U(n))$ (also see \cite{Ku10}).
In particular, this shows that 
\begin{align*}
\dim G_{2}(\mathbb{C}^{n+2})=\dim SU(n+2)/S(U(2)\times U(n))=4n.
\end{align*}
Since a maximal torus of $SU(n+2)$ is isomorphic to $T^{n+1}$, 
there is a restricted $T^{n+1}$-action on $G_{2}(\mathbb{C}^{n+2})$ and its one-skeleton has the structure of a graph (see \cite{GHZ}).
We denote this action as $(SU(n+2)/S(U(2)\times U(n)),T^{n+1})$.
Note that the action $(SU(n+2)/S(U(2)\times U(n)),T^{n+1})$ is not effective because there is the non-trivial center in $SU(n+2)$ (isomorphic to $\mathbb{Z}/(n+2)\mathbb{Z}$); therefore, the GKM graph obtained from this action does not satisfy the condition (4) in Section~\ref{sect:2}, 
i.e., the axial function is not an effective axial function.
So, in this paper, 
we define the $T^{n+1}$-action on $G_{2}(\mathbb{C}^{n+2})$ by the induced action from  
the standarad $T^{n+1}$-action on the first $(n+1)$-coordinates in $\mathbb{C}^{n+2}$ (see \eqref{Grass}). 
We denote this action as $(G_{2}(\mathbb{C}^{n+2}),T^{n+1})$.
It is easy to check that $(G_{2}(\mathbb{C}^{n+2}),T^{n+1})$ is effective and preserves the complex structure 
of $G_{2}(\mathbb{C}^{n+2})$ induced from that of $\mathbb{C}^{n+2}$. 
For example, when $n=1$, $(G_{2}(\mathbb{C}^{3}),T^{2})$ is equivariantly diffeomorphic to the complex projective space $\mathbb{C} P^2$ with the standard $T^2$-action, i.e., the toric manifold.
Note that, for $n\ge 2$, $(G_{2}(\mathbb{C}^{n+2}),T^{n+1})$ is not a toric manifold.
We also have that   
$(G_{2}(\mathbb{C}^{n+2}),T^{n+1})$ is {\it essentially isomorphic} to $(SU(n+2)/S(U(2)\times U(n)),T^{n+1})$, i.e., 
the induced effective action from $(SU(n+2)/S(U(2)\times U(n)),T^{n+1})$ is $T^{n+1}$-equivariantly diffeomorphic to $(G_{2}(\mathbb{C}^{n+2}),T^{n+1})$ up to automorphism of $T^{n+1}$ (see \cite[Definition 2.6]{Ku10} for details).
This implies that $(G_{2}(\mathbb{C}^{n+2}),T^{n+1})$ is a GKM manifold defined in Section \ref{sect:4.1} and its GKM graph satisfies the effectiveness condition (4).
\begin{remark}
GKM graphs obtained from the non-effective torus actions for flag manifolds are studied by Tymoczko in~\cite{Ty} or Fukukawa-Ishida-Masuda in~\cite{FIM} etc.
\end{remark}

In the next subsection, 
we compute the GKM graph of $(G_{2}(\mathbb{C}^{n+2}),T^{n+1})$.
For simplicity, we put 
\begin{align*}
M_{n}=G_{2}(\mathbb{C}^{n+2})
\end{align*} 
from the next subsection.

\subsection{The GKM graph of $(G_{2}(\mathbb{C}^{n+2}),T^{n+1})$}
\label{sect:5.1}

Let $(\Gamma_{n},\alpha_{n},\nabla_{n})$ be the induced GKM graph from $(M_{n},T^{n+1})$.
Note that $\Gamma_{n}=(V(\Gamma_{n}),E(\Gamma_{n}))$ is a $2n$-valent graph, because the real dimension of $M_{n}$ is $4n$, where $n\ge 1$.

We first consider the fixed points of $(M_{n},T^{n+1})$.
By definition, the Grassmannian $M_{n}$ may be identified with the following set:
\begin{align*}
\{[v_{1},v_{2}]\ |\ v_{1}, v_{2}\ \text{are linearly independent in}\ \mathbb{C}^{n+2} \},
\end{align*}
where the symbol $[v_{1},v_{2}]$ represents the equivalence class such that $[v_{1},v_{2}]$ is identified with $[w_{1},w_{2}]$ 
if two pairs of vectors $\{v_{1},v_{2}\}$ and $\{w_{1},w_{2}\}$ span the same $2$-dimensional complex vector space in $\mathbb{C}^{n+2}$.
Under this identification, the element $t\in T^{n+1}$ acts on $[v_{1},v_{2}]\in M_{n}$ by 
\begin{align*}
t\cdot [v_{1},v_{2}]\mapsto [tv_{1},tv_{2}],
\end{align*} 
where $t\in T^{n+1}$ acts on $v\in \mathbb{C}^{n+2}$ by the standard coordinatewise multiplication on the first $(n+1)$-coordinates.
Then, the fixed points can be denoted by  
\begin{align*}
M^{T}_{n}:=\{[e_{i},e_{j}]\ |\ i\not=j,\ i,j=1,\ldots, n+2\},
\end{align*} 
where $e_{1},\ldots,e_{n+2}$ are the standard basis in $\mathbb{C}^{n+2}$.
By identifying the element $[e_{i},e_{j}]\in M^{T}_{n}$ as the subset $\{i,j\}$ in $[n+2]:=\{1,2,\ldots,n+2\}$, 
we may regard the set of vertices $V(\Gamma_{n})$ as 
\begin{align*}
V(\Gamma_{n})=\{\{i,j\}\subset [n+2]\ |\ i\not=j \}.
\end{align*}
We also have that
\begin{align*}
|V(\Gamma_{n})|=
{n+2 \choose 2}
=\frac{(n+2)(n+1)}{2}.
\end{align*}

We next consider the invariant $2$-spheres in $(M_{n},T^{n+1})$.
Fix $\{i,j\}\subset [n+2]$.
Now the following subsets are $T^{n+1}$-invariant sets in $M_{n}$ which contain $[e_{i},e_{j}]$:
\begin{align*}
& S_{i,j}^{i,k}=\{ [e_{i}, v_{jk}]\in M_{n}\ |\ v_{jk}=a_{j}e_{j}+a_{k}e_{k},\ (a_{j},a_{k})\in \mathbb{C}^{2}\setminus \{0\}\}; \\
& S_{i,j}^{j,k}=\{ [v_{ik}, e_{j}]\in M_{n}\ |\ v_{ik}=a_{i}e_{i}+a_{k}e_{k},\ (a_{i},a_{k})\in \mathbb{C}^{2}\setminus \{0\}\},
\end{align*}
for all $k\in [n+2]\setminus \{i,j\}$.

Because $[e_{i},v_{jk}]$ and $[e_{i},\lambda v_{jk}]$ (for all $\lambda\in \mathbb{C}^{*}$) are the same element in $M_{n}$, we have that 
$S_{i,j}^{i,k}$ is diffeomorphic to $\mathbb{C} P^{1}$.
Similarly, $S_{i,j}^{j,k}$ is also diffeomorphic to $\mathbb{C} P^{1}$.
Moreover, $[e_{i},e_{j}], [e_{i},e_{k}]\in S_{i,j}^{i,k}$ and $[e_{i},e_{j}], [e_{k},e_{j}]\in S_{i,j}^{j,k}$.
This shows that if $\{i,j\}\cap \{k,l\}\not=\emptyset$ then the fixed points $[e_{i},e_{j}]$ and $[e_{k},e_{l}]$ are on the same invariant $2$-sphere.
Namely, the pair of two distinct sets $\{i,j\}$ and $\{k,l\}$ such that $\{i,j\}\cap \{k,l\}\not=\emptyset$ may be regarded as an edge of the  GKM graph, i.e., an element in $E(\Gamma_{n})$.
We call the edge corresponding to $S_{i,j}^{i,k}$ (resp.\ $S_{i,j}^{j,k}$) as $E_{i,j}^{i,k}\in E(\Gamma_{n})$ (resp.\ $E_{i,j}^{j,k}\in E(\Gamma_{n})$).
Note that for all $k\in [n+2]\setminus \{i,j\}$, $E_{i,j}^{i,k}$ and $E_{i,j}^{j,k}$ are out-going edges from $\{i,j\}$.
Since $\dim M_{n}=4n$, the number of out-going edges from $\{i,j\}$ is $2n$.
Hence, the set of all out-going edges from $\{i,j\}$ can be denoted by 
\begin{align*}
E_{\{i,j\}}(\Gamma_{n})=\{
E_{i,j}^{i,k}, E_{i,j}^{j,k}\ |\ k\in [n+2]\setminus \{i,j\} \}.
\end{align*}
Figure~\ref{example2-fig} shows the one-skeleton induced from $G_{2}(\mathbb{C}^{4})$.
\begin{figure}[h]
\begin{center}
\includegraphics[width=150pt,clip]{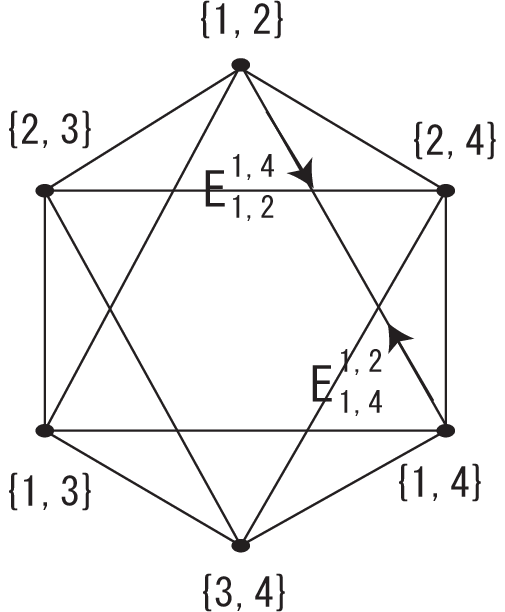}
\caption{The vertices and edges of the one-skeleton of the Grassmannian $G_{2}(\mathbb{C}^{4})$.}
\label{example2-fig}
\end{center}
\end{figure}

Next we consider the tangential representations around fixed points.
To do that, we use the following notations:
\begin{itemize}
\item the symbol $E(\eta)$ represents the total space of the fibre bundle $\eta$ over $M_{n}$;
\item the symbol $\eta_{p}$ is the restriction of $\eta$ onto $p\in M_{n}$;
\end{itemize}
Recall the structure of the tangent bundle $\tau$ of $M_{n}$.
Let $\epsilon_{\mathbb{C}}^{n+2}$ be the trivial bundle $E(\epsilon_{\mathbb{C}}^{n+2})=M_{n}\times \mathbb{C}^{n+2}\to M_{n}$.
Then, the {\it tautological vector bundle} $\gamma$ over $M_{n}$ is defined as follows:
\begin{align*}
E(\gamma)=\{(V,x)\in M_{n}\times \mathbb{C}^{n+2}\ |\ x\in V\}\to M_{n},
\end{align*}
where the projection of the bundle is just the projection onto the 1st factor.
Note that $\gamma$ is a complex $2$-dimensional vector bundle over $M_{n}$ and 
the diagonal $T^{n+1}$-action on $M_{n}\times \mathbb{C}^{n+2}$ induces the $T^{n+1}$-action on $E(\gamma)$; thus we may regard $\gamma$ as the $T^{n+1}$-equivariant vector bundle.
Let $\gamma^{\perp}$ be the normal bundle of $\gamma$ in $\epsilon_{\mathbb{C}}^{n+2}$ (we define the inner product on $\mathbb{C}^{n+2}$ as the standard Hermitian inner product). 
Since $\gamma$ is a complex $2$-dimensional vector bundle, 
$\gamma^{\perp}$ is a complex $n$-dimensional vector bundle.
Moreover, since the $T^{n+1}$-action on $\mathbb{C}^{n+2}$ preserves the standard Hermitian inner product,
the diagonal $T^{n+1}$-action on $M_{n}\times \mathbb{C}^{n+2}$ induces the $T^{n+1}$-action on $\gamma^{\perp}$.

Similar to the case of real Grassmannian (see \cite[Section 5 or proof of Theorem 14.10]{MiSt}), 
the tangent bundle $\tau$ of $M_{n}$ is isomorphic to the complex $2n$-dimensional vector bundle ${\rm Hom}(\gamma,\gamma^{\perp})$.
Therefore, the tangent space around $[e_{i},e_{j}]\in M_{n}^{T}$ may be regarded as 
\begin{align*}
\tau_{[e_{i},e_{j}]}\equiv {\rm Hom}(\gamma_{[e_{i},e_{j}]},\gamma_{[e_{i},e_{j}]}^{\perp}).
\end{align*}
Because the total space of $\gamma_{[e_{i},e_{j}]}$ is $V_{ij}:=\{A_{i}e_{i}+A_{j}e_{j}\ |\ (A_{i},A_{j})\in \mathbb{C}^{2}\}$,
its normal space in $\mathbb{C}^{n+2}$ consists of   
\begin{align*}
V_{ij}^{\perp}=\{\sum_{k\in [n+2]\setminus \{i,j\}}B_{k}e_{k}\ |\ B_{k}\in \mathbb{C}\}.
\end{align*}
Therefore, $\varphi\in {\rm Hom}(V_{ij},V_{ij}^{\perp})$ can be denoted as
\begin{align*}
\varphi(A_{i}e_{i}+A_{j}e_{j})=\sum_{k\in [n+2]\setminus \{i,j\}}f_{k}(A_{i},A_{j})e_{k}
\end{align*}
for some linear map $f_{k}:\mathbb{C}^{2}\to \mathbb{C}$, i.e., 
$f_{k}(A_{i},A_{j})=A_{i}\ell_{ik}+A_{j}\ell_{jk}$ for some $(\ell_{ik},\ell_{jk})\in \mathbb{C}^{2}$ (we will identify $f_{k}$ as $(\ell_{ik},\ell_{jk})$). 
Then, we may regard $\varphi=(f_{k})_{k\in [n+2]\setminus \{i,j\}}\in M(2,n;\mathbb{C})$ as the complex $(2\times n)$-matrix.
Now the $T^{n+1}$-actions on $\gamma$ and $\gamma^{\perp}$ induce the $T^{n+1}$-action on ${\rm Hom}(\gamma,\gamma^{\perp})$ as follows:
for $\varphi\in {\rm Hom}(\gamma_{x},\gamma^{\perp}_{x})$ ($x\in M_{n}$) and $t\in T^{n+1}$, 
\begin{align*}
t\cdot \varphi=t\circ \varphi \circ t^{-1}: \gamma_{tx}\stackrel{t^{-1}}{\longrightarrow} \gamma_{x} \stackrel{\varphi}{\longrightarrow}
\gamma_{x}^{\perp}\stackrel{t}{\longrightarrow} \gamma_{tx}^{\perp}.
\end{align*}
Therefore, on $x=[e_{i},e_{j}]$, we have $t\cdot f_{k}=(t^{-1}_{i}t_{k}\ell_{ik}\ t^{-1}_{j}t_{k}\ell_{jk})$ for $f_{k}=(\ell_{ik}\ \ell_{jk})$, $\varphi=(f_{k})\in M(2,n;\mathbb{C})$ and $t=(t_{1},\ldots, t_{n+1},1)\in T^{n+2}$, i.e., $t_{n+2}=1$.
Hence, 
on the fixed point $[e_{i},e_{j}]\in M_{n}^{T}$, we have the tangential representation as follows:
\begin{align}
\label{weights}
{\rm Hom}(\gamma_{[e_{i},e_{j}]},\gamma_{[e_{i},e_{j}]}^{\perp})\simeq \bigoplus_{k\in [n+2]\setminus \{i,j\}}V(-a_{i}+a_{k})\oplus V(-a_{j}+a_{k}),
\end{align}
where  
$a_{1},\ldots, a_{n+1}$ are the (dual) basis of the dual of Lie algebra $\algt^{*}$ of $T^{n+1}$ and we put $a_{n+2}=0$. 
It is easy to check that the factor $V(-a_{i}+a_{k})$ (resp.\ $V(-a_{j}+a_{k})$) in \eqref{weights} may be regarded as the tangent space on $[e_{i},e_{j}]$ of the invariant $2$-sphere $S_{i,j}^{j,k}$ (resp.\ $S_{i,j}^{i,k}$).
Therefore, the axial function $\alpha_{n}:E(\Gamma_{n})\to H^{2}(BT^{n+1})\simeq \algt^{*}_{\mathbb{Z}}$ is defined as follows:
\begin{align}
\label{axial-fct-on-grass}
\alpha_{n}(E_{i,j}^{i,k})=-a_{j}+a_{k},\quad \alpha_{n}(E_{i,j}^{j,k})=-a_{i}+a_{k}.
\end{align}
By the definition of edges, the orientation reverse edge satisfies $\overline{E_{i,j}^{i,k}}=E_{i,k}^{i,j}$ (resp.\ $\overline{E_{i,j}^{j,k}}=E_{j,k}^{i,j}$). 
Therefore, by the definition of the axial function the following equation holds: 
\begin{align*}
\alpha_{n}(E_{i,k}^{i,j})=-a_{k}+a_{j}=-\alpha_{n}(E_{i,j}^{i,k})\quad ({\rm resp}.\ \alpha_{n}(E_{j,k}^{i,j})=-a_{k}+a_{i}=-\alpha_{n}(E_{i,j}^{j,k})).
\end{align*}

We finally compute a connection on $(\Gamma_{n},\alpha_{n})$.
Put the connection on the edge $E_{i,j}^{i,k}$ as $(\nabla_{n})_{E_{i,j}^{i,k}}=(\nabla_{n})_{i,j}^{i,k}$.
Namely,
\begin{align*}
(\nabla_{n})_{i,j}^{i,k}:E_{\{i,j\}}(\Gamma_{n})\to E_{\{i,k\}}(\Gamma_{n}),
\end{align*}
where 
\begin{align*}
& E_{\{i,j\}}(\Gamma_{n}):= \{
E_{i,j}^{i,l}, E_{i,j}^{j,l}\ |\ l\in [n+2]\setminus \{i,j\} \},\\
& E_{\{i,k\}}(\Gamma_{n}):= \{
E_{i,k}^{i,l}, E_{i,k}^{k,l}\ |\ l\in [n+2]\setminus \{i,k\} \}.
\end{align*}
Note that the set of the weights $\{-a_{i}+a_{k}, -a_{j}+a_{k}\ |\ k\in [n+2]\setminus \{i,j\}\}$
are $3$-independent for all $\{i,j\}\subset [n+2]$ (see \eqref{axial-fct-on-grass}).
Therefore, it follows from Lemma \ref{lem:2-3} that the connection $\nabla_{n}$ on $(\Gamma_{n},\alpha_{n})$ is unique.
This implies that the bijection $(\nabla_{n})_{i,j}^{i,k}$ which satisfies the congruence relation \eqref{cong-rel} is unique.
Hence, by computing the congruence relation \eqref{cong-rel} of $(\Gamma_{n},\alpha_{n})$, the connection must be defined as follows: 
\begin{align*}
& (\nabla_{n})_{i,j}^{i,k}(E_{i,j}^{i,k})=\overline{E_{i,j}^{i,k}}=E_{i,k}^{i,j}, \\
& (\nabla_{n})_{i,j}^{i,k}(E_{i,j}^{i,l})=E_{i,k}^{i,l}\quad \text{for}\ l\in [n+2]\setminus \{i,j,k\}, \\ 
& (\nabla_{n})_{i,j}^{i,k}(E_{i,j}^{j,l})=E_{i,k}^{k,l}\quad \text{for}\ l\in [n+2]\setminus \{i,j,k\}, \\
& (\nabla_{n})_{i,j}^{i,k}(E_{i,j}^{j,k})=E_{i,k}^{j,k}.
\end{align*}
In addition, we also have that 
\begin{align*}
\alpha_{n}(E_{i,k}^{i,l})-\alpha_{n}(E_{i,j}^{i,l})=c_{i,j}^{i,k}(E_{i,j}^{i,l})\alpha_{n}(E_{i,j}^{i,k}) \\
-a_{k}+a_{l}-(-a_{j}+a_{l})=c_{i,j}^{i,k}(E_{i,j}^{i,l})(-a_{j}+a_{k})
\end{align*}
and 
\begin{align*}
\alpha_{n}(E_{i,k}^{k,l})-\alpha_{n}(E_{i,j}^{j,l})=c_{i,j}^{i,k}(E_{i,j}^{j,l})\alpha_{n}(E_{i,j}^{i,k}) \\
-a_{i}+a_{l}-(-a_{i}+a_{l})=c_{i,j}^{i,k}(E_{i,j}^{j,l})(-a_{j}+a_{k}),
\end{align*}
for $l\in [n+2]\setminus\{i,j,k\}$ and  
\begin{align*}
\alpha_{n}(E_{i,k}^{j,k})-\alpha_{n}(E_{i,j}^{j,k})=c_{i,j}^{i,k}(E_{i,j}^{j,k})\alpha_{n}(E_{i,j}^{i,k}) \\
-a_{i}+a_{j}-(-a_{i}+a_{k})=c_{i,j}^{i,k}(E_{i,j}^{j,k})(-a_{j}+a_{k}),
\end{align*}
for
some integers (congruence coefficients) $c_{i,j}^{i,k}(E_{i,j}^{i,l}), c_{i,j}^{i,k}(E_{i,j}^{j,l})$.
Therefore, 
together with Lemma~\ref{lem:2-4}, 
the congruence coefficients are 
\begin{align*}
& c_{i,j}^{i,k}(E_{i,j}^{i,k})=-2, \\ 
& c_{i,j}^{i,k}(E_{i,j}^{i,l})=-1 \quad {\text for}\ l\in [n+2]\setminus \{i,j,k\}, \\
& c_{i,j}^{i,k}(E_{i,j}^{j,l})=0\quad {\text for}\ l\in [n+2]\setminus \{i,j,k\}, \\ 
& c_{i,j}^{i,k}(E_{i,j}^{j,k})=-1.
\end{align*}

In summary we have that 
\begin{proposition}
\label{prop:5.2}
Let $\Gamma_{n}=(V(\Gamma_{n}),E(\Gamma_{n}))$ be the abstract graph defined by 
\begin{itemize}
\item the set of vertices $V(\Gamma_{n})$ consists of all $\{i,j\}$ in $[n+2]$ for $i\not=j$;
\item the set of edges $E(\Gamma_{n})$ consists of all pairs of distinct vertices $\{i,j\}$, $\{k,l\}$ such that $\{i,j\}\cap \{k,l\}\not=\emptyset$.
\end{itemize}
Define its axial function as $\alpha_{n}:E(\Gamma_{n})\to H^{2}(BT^{n+1})$ in \eqref{axial-fct-on-grass}.
Then, the connection $\nabla_{n}$ is uniquely determined (as above) and 
the triple $(\Gamma_{n},\alpha_{n},\nabla_{n})$ is the $(2n,n+1)$-type GKM graph.
\end{proposition}

\begin{remark}
The graph in Proposition~\ref{prop:5.2} is known as the {\it Johnson graph} $J(n+2,2)$.
The 1st GKM graph in Figure~\ref{examples} shows the case when $n=1$, i.e., the Johnson graph $J(3,2)$, and 
the GKM graph in Figure~\ref{example3-fig} 
 shows the case when $n=2$, i.e., the Johnson graph $J(4,2)$.
It is known that the one-skeleton of the general Grassmannian $G_{k}(\mathbb{C}^{n+k})$ (for $k\ge 1$) is the Johnson graph $J(n+k,k)$ (see \cite{BGH}).
\end{remark}

\begin{figure}[h]
\begin{center}
\includegraphics[width=250pt,clip]{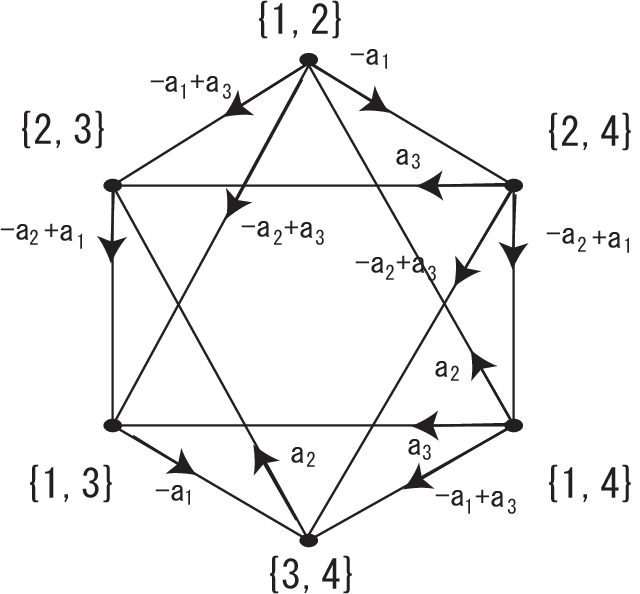}
\caption{The GKM graph $(\Gamma_{2},\alpha_{2},\nabla_{2})$ of $(G_{2}(\mathbb{C}^4), T^3)$.}
\label{example3-fig}
\end{center}
\end{figure}

\subsection{The 2nd main result}
\label{sect:5.2}

Because we fix the axial function $\alpha_{n}$ on $\Gamma_{n}$ and its connection $\nabla_{n}$ is unique, 
we may write the GKM graph $(\Gamma_{n},\alpha_{n},\nabla_{n})$ of $(M_{n},T^{n+1})$ as $\Gamma_{n}$ for simplicity; therefore, we denote 
the group of axial functions $\mathcal{A}(\Gamma_{n},\alpha_{n},\nabla_{n})$ as $\mathcal{A}(\Gamma_{n})$.
This final section is devoted to the proof of the following theorem:
\begin{theorem}
\label{main:final}
The group of axial functions $\mathcal{A}(\Gamma_{n})$ is isomorphic to $\mathbb{Z}^{n+1}$.
\end{theorem}
When $n=1$, the GKM graph $\Gamma_{1}$ is the $(2,2)$-type GKM graph (which is the 1st GKM graph in Figure~\ref{examples}).
Therefore, by Theorem~\ref{thm-main}, we have that $\mathcal{A}(\Gamma_{1})\simeq \mathbb{Z}^{2}$.
Hence, we may assume that $n\ge 2$.

To prove Theorem~\ref{main:final}, 
we first choose an order on $E_{\{i,j\}}(\Gamma_{n})$ for $i,j \in [n+2]$ as follows (see Figure~\ref{order-graph} for $n=2$):
\begin{itemize}
\item $E_{i,j}^{i,k}\prec E_{i,j}^{j,l}$ if $i<j$, where $k,l\in [n+2]\setminus \{i,j\}$;
\item $E_{i,j}^{i,k}\prec E_{i,j}^{i,l}$ if $k<l$, where $k,l\in [n+2]\setminus \{i,j\}$.
\end{itemize}
\begin{figure}[h]
\begin{center}
\includegraphics[width=150pt,clip]{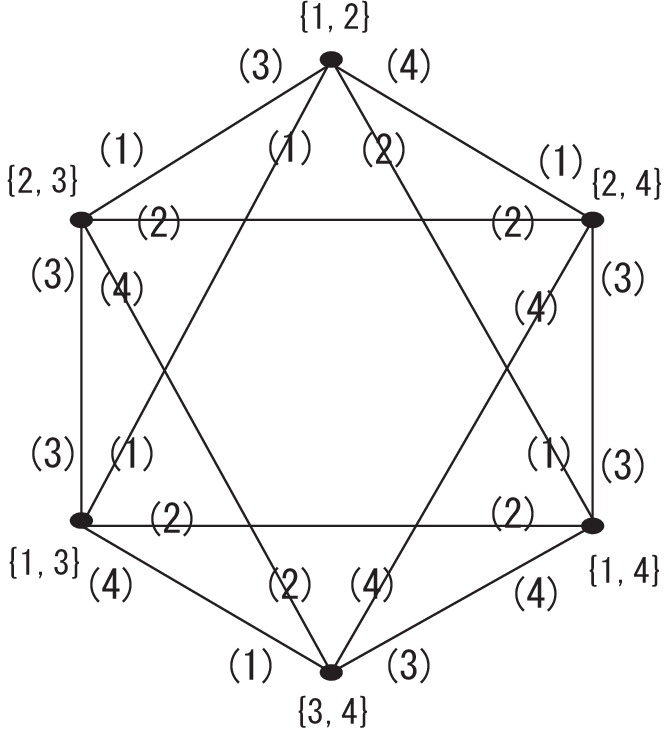}
\caption{The order of out-going edges on each vertex of $\Gamma_{2}$. For example, around the vertex $\{1,3\}$, Figure shows that 
$E_{1,3}^{1,2}\prec E_{1,3}^{1,4}\prec E_{1,3}^{2,3}\prec E_{1,3}^{3,4}$.}
\label{order-graph}
\end{center}
\end{figure}

Take $f\in \mathcal{A}(\Gamma_{n})$ and put 
\begin{align*}
f(\{n+1,n+2\})=
\begin{pmatrix}
x_{1} \\
\vdots \\
x_{2n}
\end{pmatrix}
\end{align*}
with respect to the order on $E_{\{n+1,n+2\}}(\Gamma_{n})$ defined as before, i.e.,
\begin{align*}
E^{1,n+1}_{n+1,n+2}\prec \cdots \prec E^{n,n+1}_{n+1,n+2}\prec E^{1,n+2}_{n+1,n+2} \prec\cdots \prec E^{n,n+2}_{n+1,n+2}. 
\end{align*}
More precisely, using the notation $f(p)_{e}$ defined in Section~\ref{sect:2.3} for $p\in V(\Gamma_{n})$ and $e\in E_{p}(\Gamma_{n})$,
we define the following correspondence between edges and integers (variables):
\begin{align*}
\begin{array}{rcl}
E_{n+1,n+2}^{1,n+1} & \mapsto & f(\{n+1,n+2\})_{E_{n+1,n+2}^{1,n+1}}=x_{1}; \\
& \vdots & \\
E_{n+1,n+2}^{n,n+1} & \mapsto & f(\{n+1,n+2\})_{E_{n+1,n+2}^{n,n+1}}=x_{n}; \\
E_{n+1,n+2}^{1,n+2} & \mapsto & f(\{n+1,n+2\})_{E_{n+1,n+2}^{1,n+2}}=x_{n+1}; \\
& \vdots & \\
E_{n+1,n+2}^{n,n+2} & \mapsto & f(\{n+1,n+2\})_{E_{n+1,n+2}^{n,n+2}}=x_{2n}.
\end{array}
\end{align*} 
Then, by the connectedness of $\Gamma_{n}$ and the definition of $\mathcal{A}(\Gamma_{n})$, 
the vector $f(\{i,j\})$ is denoted by the variables $x_{1},\ldots, x_{2n}$ for all $\{i,j\}$'s in $V(\Gamma_{n})$.
This shows that ${\rm rk}\ \mathcal{A}(\Gamma_{n})\le 2n$ (this is also known from Theorem~\ref{thm-main} and the fact that $\Gamma_{n}$ is a $(2n,n+1)$-type GKM graph).
By Theorem~\ref{thm-main}, we also have ${\rm rk}\ \mathcal{A}(\Gamma_{n})\ge n+1$. 
Therefore, in order to prove Theorem~\ref{main:final}, 
it is enough to prove that the variables $x_{n+2},\ldots, x_{2n}$ can be denoted by the other variables 
$x_{1},\ldots, x_{n+1}$.
We shall prove that the following lemma holds. 
\begin{lemma}\label{finallem}
For $j=0,\ldots, n-2$ $(n\ge 2)$, the following equation holds:
\begin{align*}
x_{2n-j}=-x_{1}+x_{n-j}+x_{n+1}.
\end{align*}
\end{lemma}
\begin{proof}
Recall the definition of the connection $\nabla_{n}$ in Section~\ref{sect:5.1}.
There is the triangle GKM subgraph in $\Gamma_{n}$ (i.e., the subgraph closed under the connection) 
which spanned by the vertices $\{n+1,n+2\},\ \{1,n+1\}, \{1,n+2\}$ (see Figure~\ref{key-arguments1}).
We first show the variables corresponding edges in this triangle as in Figure~\ref{key-arguments1}.
\begin{figure}[h]
\begin{center}
\includegraphics[width=170pt,clip]{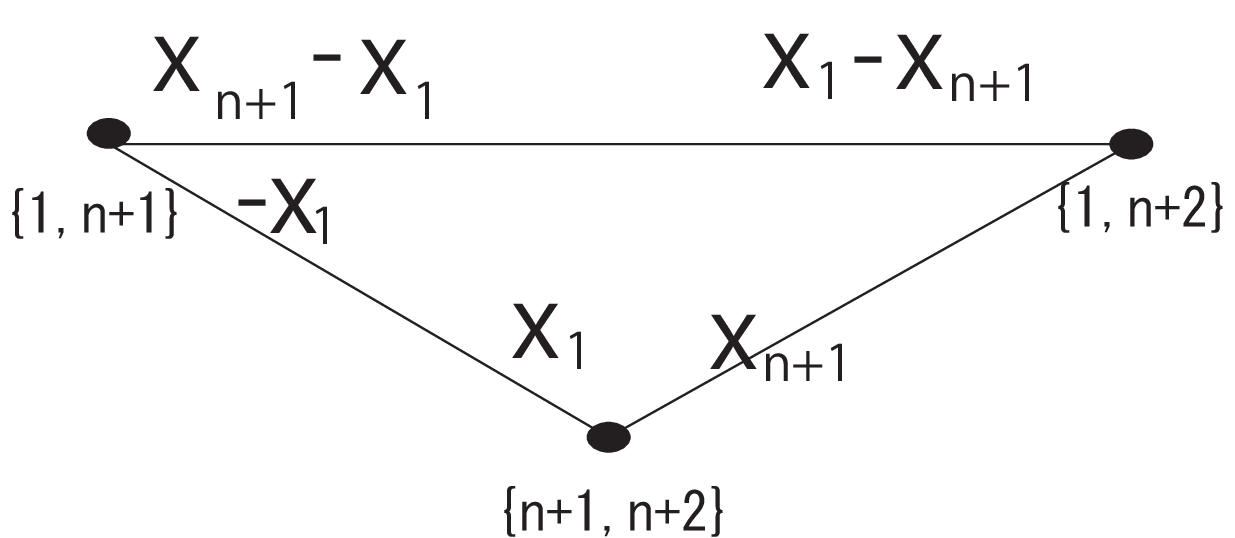}
\caption{The triangle GKM subgraph with corresponding variables on edges.}
\label{key-arguments1}
\end{center}
\end{figure}

We assumed $f(\{n+1,n+2\})_{E_{n+1,n+2}^{1,n+1}}=x_{1}$ (and $f(\{n+1,n+2\})_{E_{n+1,n+2}^{1,n+2}}=x_{n+1}$).
So, by Lemma~\ref{lem:2-12}, 
\begin{align*}
f(\{1,n+1\})_{E_{1,n+1}^{n+1,n+2}}=-x_{1}.
\end{align*}
Moreover, the connection 
\begin{align*}
(\nabla_{n})_{n+1,n+2}^{1,n+1}(E_{n+1,n+2}^{1,n+2})=E_{1,n+1}^{1,n+2}
\end{align*} 
and 
the congruence coefficient 
\begin{align*}
c_{n+1,n+2}^{1,n+1}(E_{n+1,n+2}^{1,n+2})=-1.
\end{align*}
Therefore, we have the following equation by the definition of $f\in \mathcal{A}(\Gamma_{n})$:
\begin{align*}
x_{n+1}-f(\{1,n+1\})_{E_{1,n+1}^{1,n+2}}=(-1)\times (-x_{1}).
\end{align*}
Hence, we have $f(\{1,n+1\})_{E_{1,n+1}^{1,n+2}}=x_{n+1}-x_{1}$, i.e., 
the correspondence $E_{1,n+1}^{1,n+2}\mapsto x_{n+1}-x_{1}$.
Together with  Lemma~\ref{lem:2-12} we also have the correspondence 
\begin{align*}
E_{1,n+2}^{1,n+1}\mapsto x_{1}-x_{n+1}.
\end{align*}
This establishes the variables in Figure~\ref{key-arguments1}.

We next consider the subgraph drawn in Figure~\ref{key-arguments2} and compute the corresponding variables on edges in this subgraph as in Figure~\ref{key-arguments2}.
\begin{figure}[h]
\begin{center}
\includegraphics[width=180pt,clip]{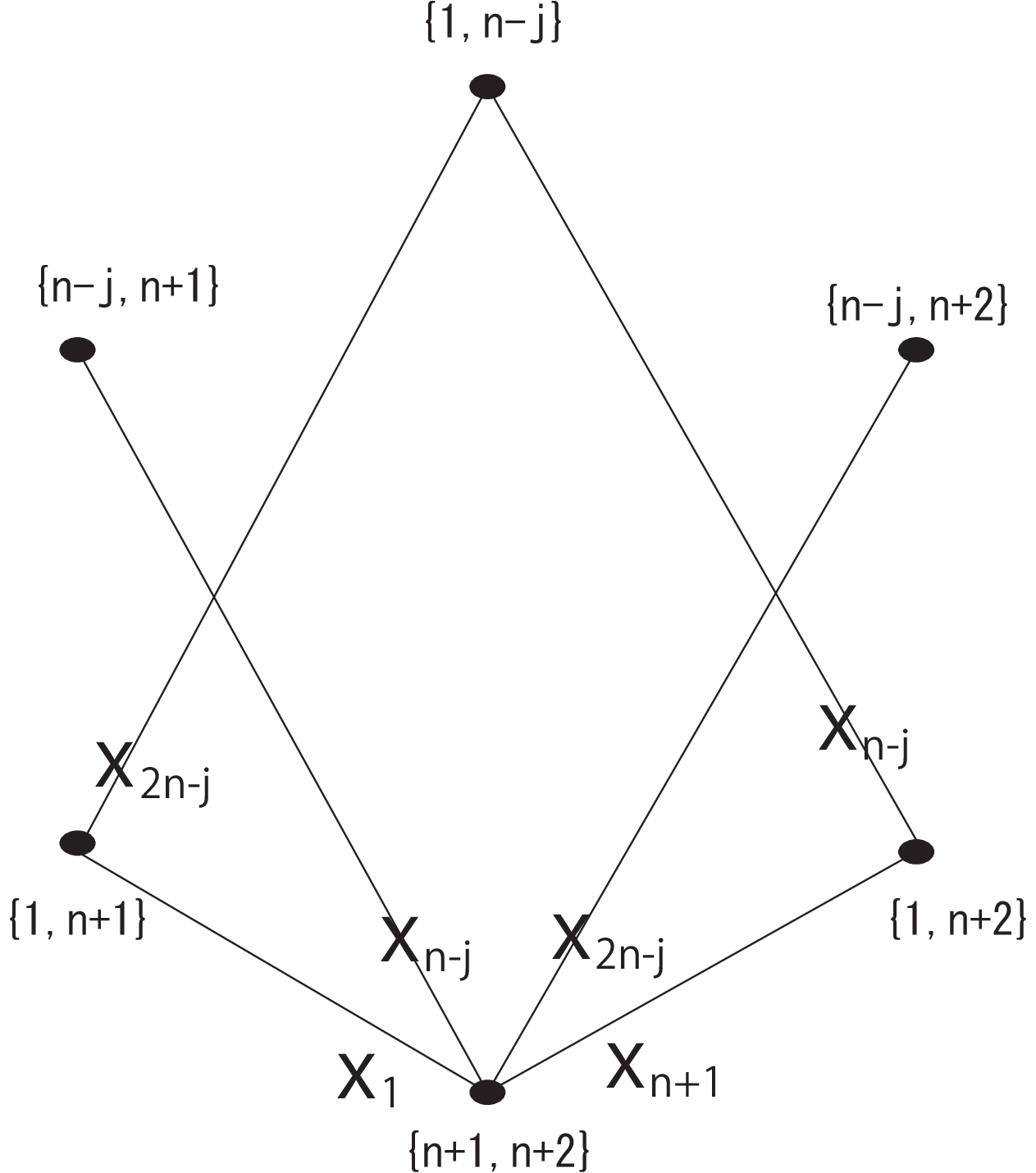}
\caption{The subgraph with corresponding variables on edges.}
\label{key-arguments2}
\end{center}
\end{figure}

We assumed $f(\{n+1,n+2\})_{E_{n+1,n+2}^{n-j,n+2}}=x_{2n-j}$ (and $f(\{n+1,n+2\})_{E_{n+1,n+2}^{n-j,n+1}}=x_{n-j}$) for $0\le j\le n-2$.
Because $(\nabla_{n})_{n+1,n+2}^{1,n+1}(E_{n+1,n+2}^{n-j,n+2})=E_{1,n+1}^{1,n-j}$ and $c_{n+1,n+2}^{1,n+1}(E_{n+1,n+2}^{n-j,n+2})=0$,
we have the correspondence 
\begin{align*}
E_{1,n+1}^{1,n-j}\mapsto x_{2n-j}.
\end{align*}
Similarly, because $(\nabla_{n})_{n+1,n+2}^{1,n+2}(E_{n+1,n+2}^{n-j,n+1})=E_{1,n+2}^{1,n-j}$ and $c_{n+1,n+2}^{1,n+2}(E_{n+1,n+2}^{n-j,n+2})=0$,
we have the correspondence 
\begin{align*}
E_{1,n+2}^{1,n-j}\mapsto x_{n-j}.
\end{align*}
This establishes the variables in Figure~\ref{key-arguments2}.

By Figure~\ref{key-arguments1} and Figure~\ref{key-arguments2}, 
we have the triangle GKM subgraph with variables as in Figure~\ref{key-arguments3}.
\begin{figure}[h]
\begin{center}
\includegraphics[width=150pt,clip]{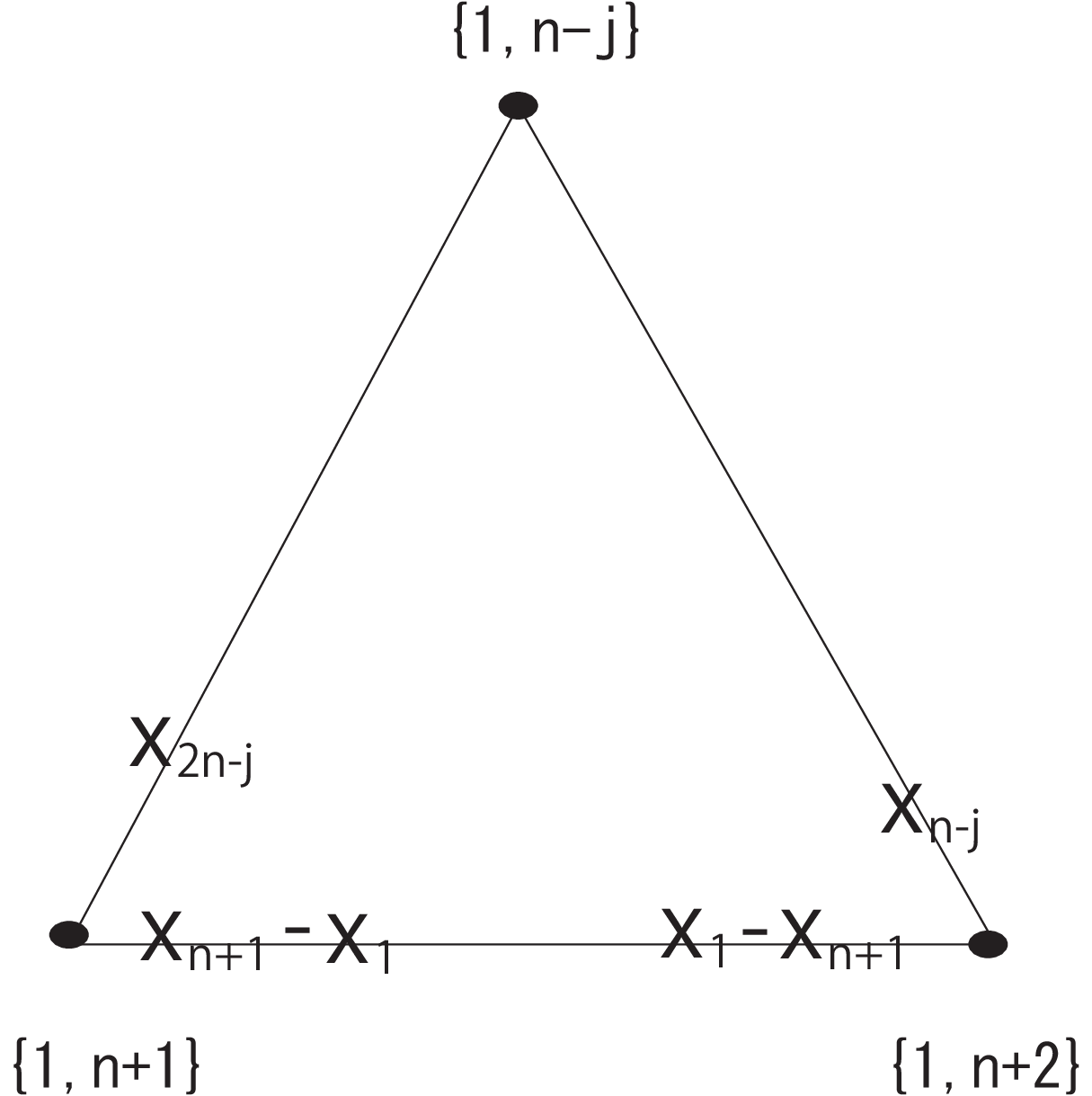}
\caption{The triangle GKM subgraph with corresponding variables on edges.}
\label{key-arguments3}
\end{center}
\end{figure}

In Figure~\ref{key-arguments3}, 
$(\nabla_{n})_{1,n+1}^{1,n+2}(E_{1,n+1}^{1,n-j})=E_{1,n+2}^{1,n-j}$ and $c_{1,n+1}^{1,n+2}(E_{1,n+1}^{1,n-j})=-1$.
Therefore, by definition of $f\in \mathcal{A}(\Gamma_{n})$,
we have the equation 
\begin{align*}
x_{2n-j}-x_{n-j}=-1(x_{1}-x_{n+1}).
\end{align*} 
This establishes that 
$x_{2n-j}=-x_{1}+x_{n-j}+x_{n+1}$.
\end{proof}

Consequently, this shows Theorem \ref{main:final}.
Therefore, by Corollary~\ref{main:2}, we have Proposition~\ref{main:3}.

\section*{Acknowledgement}
The author would like to thank to Dong Youp Suh, Jongbaek Song and Eunjeong Lee for their valuable comments and their careful reading.
This work was supported by JSPS KAKENHI Grant Number 15K17531, 24224002, 17K14196.

\end{document}